\documentclass[11pt]{amsart}

\usepackage{amssymb,verbatim}
\usepackage{amsmath,amsfonts}
\usepackage[mathscr]{euscript} 
\usepackage{amsthm}
\usepackage{url}
\usepackage{graphicx} 
\usepackage{enumitem} 
\usepackage{hyperref} 
\hypersetup{colorlinks} 
\usepackage{bbm} 
\usepackage{bm} 
\usepackage{mathrsfs} 
\usepackage{cancel} 

\usepackage[colorinlistoftodos]{todonotes}

\RequirePackage{cleveref}
\usepackage{hypcap}
\hypersetup{colorlinks=true, citecolor=darkblue, linkcolor=darkblue}
\definecolor{darkblue}{rgb}{0.0,0,0.7}
\newcommand{\darkblue}{\color{darkblue}}

\definecolor{darkred}{rgb}{0.68,0,0}
\newcommand{\darkred}{\color{darkred}}

\definecolor{darkgreen}{rgb}{0,.38,0}
\newcommand{\darkgreen}{\color{darkgreen}}

\definecolor{magenta}{rgb}{.51, 0, .51}
\newcommand{\magenta}{\color{magenta}}

\newcommand{\defn}[1]{\emph{\darkblue #1}}
\newcommand{\defna}[1]{\emph{\darkred #1}}
\newcommand{\defnb}[1]{\emph{\darkblue #1}}
\newcommand{\defng}[1]{\emph{\darkgreen #1}}
\newcommand{\defnm}[1]{\emph{\magenta #1}}

\setlist[enumerate]{
	label=\textnormal{({\roman*})},
	ref={\roman*}}

\makeatletter
\def\th@plain{%
	\thm@notefont{}
	\itshape 
}
\def\th@definition{%
	\thm@notefont{}
	\normalfont 
}
\makeatother

\newtheorem{thm}{Theorem}[section]
\newtheorem{lemma}[thm]{Lemma}
\newtheorem{claim}[thm]{Claim}
\newtheorem*{claim*}{Claim}
\newtheorem{cor}[thm]{Corollary}
\newtheorem{prop}[thm]{Proposition}
\newtheorem{conj}[thm]{Conjecture}
\newtheorem{conv}[thm]{Convention}

\theoremstyle{definition}

\newtheorem{rem}[thm]{Remark}

\numberwithin{figure}{section}
\numberwithin{equation}{section}

\def\zz{\mathbb Z}

\def\nz{\mathbb \zz_{\ge 1}}

\def\nn{\mathbb N}

\def\rr{\mathbb R}
\def\qqq{\mathbb Q}

\def\pp{\mathbb P}

\def\sm{\smallsetminus}

\def\Om{\Omega}

\def\la{\lambda}

\def\de{\delta}

\def\al{\alpha}
\def\be{\beta}

\def\vp{\varphi}

\def\cC{\mathcal C}

\def\cF{\mathcal F}

\def\cO{\mathcal O}

\def\cS{\mathcal S}

\def\ssu{\subset}

\def\<{\langle}
\def\>{\rangle}

\def\TU{\text{{\rm TU}}}

\def\rI{ {\text {\rm I} } }

\def\vt{\vartheta}

\def\Q{{\text {\rm Q} } }

\def\0{{\mathbf 0}}

\def\.{\hskip.06cm}
\def\ts{\hskip.03cm}

\def\bz{\zb}

\def\bx{\xb}

\def\ba{\bxa}

\def\bc{\bxc}

\def\bbb{\bxb}

\def\di{{\small{\ts\diamond\ts}}}

\newcommand{\per}{\mathrm{per}}

\def\aM{\textrm{M}}

\def\aN{\textrm{N}}
\def\aNr{\textrm{\em N}}

\def\ag{\rho}
\def\agr{\rho}

\def\.{\hskip.06cm}
\def\ts{\hskip.03cm}

\def\nin{\noindent}

\newcommand{\textsu}[1]{\textup{\textsf{#1}}}

\newcommand{\ComCla}[1]{\textup{\textsu{#1}}}

\newcommand{\sharpP}{\ComCla{\#P}}
\newcommand{\SP}{\ComCla{\#P}}

\newcommand{\GapP}{\ComCla{GapP}}
\newcommand{\GapPP}{\ComCla{GapP}_{\ge 0}}

\newcommand{\Sigmap}{\ensuremath{\Sigma^{{\textup{p}}}}}

\newcommand{\NP}{\ComCla{NP}}
\newcommand{\BPP}{\ComCla{BPP}}
\newcommand{\coNP}{\ComCla{coNP}}
\renewcommand{\P}{\ComCla{P}}
\newcommand{\CeqP}{\ComCla{C$_=$P}}

\newcommand{\PH}{\ComCla{PH}}
\newcommand{\PSPACE}{\ComCla{PSPACE}}

\newcommand{\FP}{\ComCla{FP}}
\newcommand{\PP}{\ComCla{PP}}

\def\SP{\sharpP}
\def\poly{{\P}}
\def\CEP{{\CeqP}}

\def\xxi{\textbf{\textit{x}}}

\def\bbi{\textbf{{\textit{b}}}}

\def\Com{ {\text {\rm Com} } }
\def\lc{ {\textup {\rm lc} } }
\def\uc{ {\textup {\rm uc} } }

\newcommand{\iA}{\textnormal{A}} 
\newcommand{\iB}{\textnormal{B}} 
\newcommand{\iK}{\textnormal{K}} 
\newcommand{\iL}{\textnormal{L}} 
\newcommand{\iM}{\textnormal{M}} 
\newcommand{\iN}{\textnormal{N}} 

\newcommand{\iQ}{\textnormal{Q}} 

\newcommand{\iS}{\textnormal{S}} 
\newcommand{\iV}{\textnormal{V}} 

\newcommand{\Vol}{\textnormal{Vol}} 

\DeclareMathOperator{\cb}{\mathbf{c}} 
\DeclareMathOperator{\Ec}{\mathcal{E}} 

\DeclareMathOperator{\Rb}{\mathbb{R}} 
\DeclareMathOperator{\xb}{\mathbf{x}} 
\DeclareMathOperator{\yb}{\mathbf{y}} 
\DeclareMathOperator{\zb}{\mathbf{z}} 

\DeclareMathOperator{\bxa}{\mathbf{a}} 
\DeclareMathOperator{\bxb}{\mathbf{b}} 
\DeclareMathOperator{\bxc}{\mathbf{c}} 

\DeclareMathOperator{\Nc}{\mathcal{N}}

\DeclareMathOperator{\inc}{\text{inc}} 
\DeclareMathOperator{\comp}{\text{com}} 

\DeclareMathOperator{\EAF}{\textsc{EqualityAF}}
\DeclareMathOperator{\ESta}{\textsc{EqualityStanley}}

\DeclareMathOperator{\aF}{\textnormal{F}}
\def\aFr{\textrm{\em F}}

\title[Equality cases of the Alexandrov--Fenchel inequality]
{Equality cases of the Alexandrov--Fenchel inequality \\ are not in the polynomial hierarchy}
\date{\today}

 \author{Swee Hong Chan}
 \address[Swee Hong Chan]{Department of Mathematics, Rutgers University,  Piscatway, NJ 08854.}
 \email{\texttt{sweehong.chan@rutgers.edu}}

 \author[\ts Igor Pak]{Igor Pak}
 \address[Igor Pak]{Department of Mathematics, UCLA,  Los Angeles, CA 90095.}
 \email{\texttt{pak@math.ucla.edu}}

\begin{document}

\begin{abstract}
Describing the equality conditions of the \emph{Alexandrov--Fenchel inequality}
\ts has been a major open problem for decades.  We prove that in the case
of convex polytopes, this description is not in the polynomial hierarchy
unless the polynomial hierarchy collapses to a finite level.  This is the
first hardness result for the problem, and is a complexity counterpart of
the recent result by Shenfeld and van~Handel \cite{SvH-acta}, which
gave a geometric characterization of the equality conditions.
The proof involves Stanley's \emph{order polytopes} \ts and employs
poset theoretic technology.

\end{abstract}
	
		\keywords{Alexandrov–Fenchel inequality, equality cases, order polytope, stability, polynomial hierarchy, log-concavity, Stanley's inequality, linear extension, continued fraction}
	
	\subjclass[2020]{{05A20, 06A07, 52A40} (Primary); {11A55, 52A39, 68Q15, 68Q17, 68R05}~(Secondary).}
	
\maketitle
	
\vskip-.5cm


\section{Introduction}\label{s:intro}

\subsection{Foreword} \label{ss:intro-foreword}
Geometric inequalities play a central role in convex geometry, probability
and analysis, with numerous combinatorial and algorithmic applications.
The \defna{Alexandrov--Fenchel {\em (AF)}~inequality} \ts lies close to
the heart of convex geometry.  It is one of the deepest and most general
results in the area, generalizing a host of simpler geometric
inequalities such as the \emph{isoperimetric inequality} \ts
and the \emph{Brunn--Minkowski inequality},
see~$\S$\ref{ss:hist-geom}.

The equality conditions for geometric inequalities are just as fundamental
as the inequalities themselves, and are crucial for many applications,
see~$\S$\ref{ss:finrem-equality}.
For simpler inequalities they tend to be straightforward and follow from
the proof.  As the inequalities become more complex, their proofs became
more involved, and the equality cases become more numerous and cumbersome.
This is especially true for the Alexandrov--Fenchel inequality, where
the complete description of the equality cases remain open despite much
effort and many proofs, see~$\S$\ref{ss:hist-AF}.

We use the language and ideas from computational complexity and
tools from poset theory, to prove that the equality cases of the
Alexandrov--Fenchel inequality cannot be explicitly described
for convex polytopes in a certain formal sense.
We give several applications to stability in geometric inequalities
and to combinatorial interpretation of the defect of poset inequalities.
We also raise multiple questions, both mathematical and philosophical,
see Section~\ref{s:finrem}.


\smallskip

\subsection{Alexandrov--Fenchel inequality} \label{ss:intro-main}
Let \. $\iV(\iQ_1,\ldots, \iQ_n)$ \. denote the \emph{mixed volume} \ts
of convex bodies \. $\iQ_1,\ldots,\iQ_n$ \ts in \ts $\rr^n$ (see below).
The \defn{Alexandrov--Fenchel inequality} \ts states that for convex bodies
\ts $\iK, \iL,\iQ_1,\ldots,\iQ_{n-2}$ \ts in $\rr^{n}$, we have:
\begin{equation}\label{eq:AF} \tag{AF}
		\iV\big(\iK,\iL,\iQ_1,\ldots,\iQ_{n-2}\big)^2  \, \geq \, \iV\big(\iK,\iK,\iQ_1,\ldots,\iQ_{n-2}\big)
\.\cdot \.   \iV\big(\iL,\iL,\iQ_1,\ldots,\iQ_{n-2}\big).
\end{equation}

\smallskip

Let polytope \ts $\iK \subset \rr^n$ \ts be defined by a system of inequalities
\. $A \. \xxi \leqslant \bbi$.
We say that \ts $\iK$ \ts is a \ts \defn{$\TU$-polytope} \ts if vector \. $\bbi\in \zz^n$,
and matrix \ts $A$ \ts is  \defn{totally unimodular}, i.e.\ all its minors
have determinants in  \ts $\{0,\pm 1\}$.  Note that all vertices of TU-polytopes
are integral.  Denote by \ts $\EAF$ \ts the \defnm{equality verification
 problem of the Alexandrov--Fenchel inequality}, defined as
 the decision problem whether \eqref{eq:AF}  is an equality.

\smallskip

\begin{thm}[{\rm Main theorem}{}]\label{t:main-AF}
Let \. $\iK, \iL,\iQ_1,\ldots,\iQ_{n-2} \subset \rr^{n}$ \ts be \ts $\TU$-polytopes.
Then the equality verification problem of the Alexandrov--Fenchel inequality~\eqref{eq:AF}
is not in the polynomial hierarchy unless the polynomial hierarchy collapses
to a finite level:
$$\EAF \ts \in \PH \ \ \Longrightarrow \ \ \PH=\Sigmap_m  \quad \ \text{for some} \ \, m\ts.
$$
\end{thm}

\smallskip

Informally, the theorem says that the equality cases of the
Alexandrov--Fenchel inequality \eqref{eq:AF} are \defna{unlikely to have a
description in the polynomial hierarchy}.\footnote{The collapse in the
theorem contradicts standard assumptions in computational complexity.
A conjecture that the collapse does not happen is a strengthening of
the \ts $\P\ne \NP$ \ts conjecture that remains out of reach,
see~$\S$\ref{ss:hist-CA}.}
This is in sharp contrast with other geometric inequalities,
including many special cases of~\eqref{eq:AF},
when the equality cases have an explicit description,
thus allowing an efficient verification (see~$\S$\ref{ss:hist-geom}).

Let us emphasize that constraining to TU-polytopes makes the theorem
stronger rather than weaker.  Indeed, one would hope that the equality
verification problem is easy at least in the case when both vertices
and facets are integral (cf.~$\S$\ref{ss:finrem-polytopes}).  In fact,
we chose the \emph{smallest} \ts natural class of H-polytopes which
contains all order polytopes (see below).

Let us quickly unpack the very strong claim of Theorem~\ref{t:main-AF}.
In particular, the theorem implies that given the polytopes,
the equality in~\eqref{eq:AF} cannot be decided in polynomial time: \ts
\ts $\EAF\notin \P$, nor even in probabilistic polynomial
time: \ts $\EAF\notin\BPP$ \ts (unless \ts $\PH$ \ts collapses).
Moreover, 
there can be no polynomial size certificate which verifies
that \eqref{eq:AF} is an equality: \ts $\EAF\ts\notin \NP$, or
a strict inequality:  \ts  $\EAF\notin \coNP$ \ts (ditto).

Our results can be viewed as a complexity theoretic counterpart of the
\defng{geometric description} \ts
of the Alexandrov--Fenchel inequality that was proved recently by Shenfeld
and van~Handel \cite{SvH-acta}.  In this context, Theorem~\ref{t:main-AF} says
that this geometric description is not computationally effective, and cannot
be made so under standard complexity assumptions.  From this point of view,
the results in \cite{SvH-acta} are optimal, at least for convex polytopes
in the full generality (cf.~$\S$\ref{ss:finrem-meaning}).

\medskip

{\small
\nin
\emph{Warning}: \ts Here we only give statements of the results without
a context.  Our hands are tied by the interdisciplinary nature of the
paper with an extensive background in both convex geometry,
poset theory and computational complexity.
We postpone the definitions until Section~\ref{s:def}, and
the review until Section~\ref{s:hist}.
}

\smallskip

\subsection{Stability} \label{ss:intro-stability}
In particular, Theorem~\ref{t:main-AF} prohibits
certain \defna{stability inequalities}.  In the context of general inequalities,
these results give quantitative measurements of how close are the objects of
study (variables, surfaces, polytopes, lattice points, etc.) to the equality cases
in some suitable sense, when the inequality is close
to an equality, see e.g.\ \cite{Fig13}.

In the context of geometric inequalities, many sharp stability results
appear in the form of \defn{Bonnesen type inequality}, see \cite{Oss79}.
These are defined as the strengthening of a geometric inequality \.
$f\geqslant g$ \. to \. $f-g\geqslant h$, such that \. $h\geqslant 0$,
and \. $h=0$ \. \underline{if and only if} \. $f=g$.\footnote{Following~\cite{Oss79},
function $h$ should also have a (not formally defined) ``geometric description''. }
They are named after the celebrated extension of the \defng{isoperimetric
inequality} \ts by Bonnesen (see~$\S$\ref{ss:hist-stab}).

While there are numerous Bonnesen type inequalities of various strength
for the Brunn--Minkowski inequalities and their relatives,
the case of Alexandrov--Fenchel inequality~\eqref{eq:AF}
remains unapproachable in full generality.  Formally, define the
\defn{Alexandrov--Fenchel defect} \. as:
{\small
\begin{equation*}\label{eq:AF-stab} 
\de\big(\iK,\iL,\iQ_1,\ldots,\iQ_{n-2}\big)  \. := \.
		\iV\big(\iK,\iL,\iQ_1,\ldots,\iQ_{n-2}\big)^2 \ts - \ts
\iV\big(\iK,\iK,\iQ_1,\ldots,\iQ_{n-2}\big) \cdot
\ts \iV\big(\iL,\iL,\iQ_1,\ldots,\iQ_{n-2}\big).
\end{equation*}}

\nin
One would want to find a bound of the form \. $\de(\cdot) \geqslant \xi(\cdot)$, where \ts
$\xi$ \ts is a nonnegative computable function of the polytopes.
The following result is an easy corollary from the proof of Theorem~\ref{t:main-AF}.

\smallskip

\begin{cor}\label{c:main-AF-stab}
Suppose \. $\de\big(\iK,\iL,\iQ_1,\ldots,\iQ_{n-2}\big)  \geqslant
\xi\big(\iK,\iL,\iQ_1,\ldots,\iQ_{n-2}\big)$ \.
is a Bonnesen type inequality, such that \ts $\xi$ \ts is computable
in polynomial time on all \ts $\TU$-polytopes.  Then \ts $\PH=\NP$.
\end{cor}

\smallskip

Informally, the corollary implies that for the stability of the
AF~inequality, one should either avoid polytopes altogether
and require some regularity conditions for the convex bodies
(as has been done in the past, see~$\S$\ref{ss:hist-stab}), or be content
with functions \ts $\xi$ \ts which are hard to compute (such inequalities
can still be very useful, of course).  See~$\S$\ref{ss:finrem-mass-transport}
for further implications.

To understand how the corollary follows from the proof of Theorem~\ref{t:main-AF},
the Bonnesen condition in this case states that \. $\xi(\cdot) = 0$ \.
\underline{if and only if} \. $\de(\cdot) = 0$.
Thus, the equality \. $\{\de(\cdot) =^? 0\}$ \. can be decided in
polynomial time on TU-polytopes, giving the assumption in the theorem.

\smallskip

\subsection{Stanley inequality} \label{ss:intro-Stanley}
We restrict ourselves to a subset of TU-polytopes
given by the \defng{order polytopes} \ts (see~$\S$\ref{ss:hist-poset-polytopes}).
Famously, Stanley showed in~\cite{Sta-AF}, that
the Alexandrov--Fenchel inequality applied to certain such polytopes
gives the \defna{Stanley inequality}, that the numbers of certain linear
extensions of finite posets form a log-concave sequence. This inequality
is of independent interest in order theory (see~$\S$\ref{ss:hist-LE}),
and is the starting point of our investigation.

\smallskip

Let \. $P=(X,\prec)$ \. be a poset with \. $|X|=n$ \. elements.
Denote \. $[n]:=\{1,\ldots,n\}$.
A \defn{linear extension} of $P$ is a bijection \. $f: X \to [n]$,
such that
\. $f(x) < f(y)$ \. for all \. $x \prec y$.
Denote by \ts $\Ec(P)$ \ts the set of linear extensions of $P$,
and let \. $e(P):=|\Ec(P)|$.

Let \. $x,z_1,\ldots,z_k\in X$ \. and \. $a,c_1,\ldots,c_k\in [n]$;
we write \. $\bz =(z_1,\ldots,z_k)$ \. and \. $\bc =(c_1,\ldots,c_k)$, and we assume without loss of generality that \. $c_1<\ldots< c_k$\..
Let \.
$\Ec_{\bz\bc}(P,x,a)$ \. be the set of linear extensions \. $f\in \Ec(P)$,
such that \. $f(x)=a$ \. and \. $f(z_i)=c_i$ \. for all \. $1\le i \le k$.
Denote by \. $\aN_{\bz\bc}(P, x,a):=\bigl|\Ec_{\bz\bc}(P,x,a)\bigr|$ \.
the number of such linear extensions. The \defn{Stanley inequality}~\cite{Sta-AF} states that
the sequence \. $\big\{\aN_{\bz\bc}(P, x,a), 1\le a \le n\big\}$ \. is \defng{log-concave}:
\begin{equation}\label{eq:Sta}\tag{Sta}
\aN_{\bz\bc}(P, x,a)^2 \, \ge \, \aN_{\bz\bc}(P, x,a+1) \.\cdot \.  \aN_{\bz\ts\bc}(P, x,a-1).
\end{equation}

\smallskip

The problem of finding the equality conditions for \eqref{eq:Sta} was
first asked by Stanley in the original paper \cite[$\S$3]{Sta-AF}, see also
\cite[Question~6.3]{BT02}, \cite[$\S$9.9]{CPP-effective} and \cite{MS22}.
Formally, for every \ts $k\ge 0$, denote by \ts $\ESta_k$ \ts the
\defnm{equality verification problem of the Stanley inequality \ts
with $k$ fixed elements},
defined as the decision problem whether \eqref{eq:Sta}  is an equality.
It was shown by Shenfeld and van Handel that \ts  $\ESta_0 \in \poly$, see
 \cite[Thm~15.3]{SvH-acta}.

 \smallskip

\begin{thm}\label{t:main-Sta}
Let \ts $k\ge 2$.
Then the equality verification problem of the Stanley inequality \eqref{eq:Sta}
is not in the polynomial hierarchy unless the polynomial hierarchy collapses
to a finite level:
$$\ESta_k \ts \in \PH \ \ \Longrightarrow \ \ \PH=\Sigmap_m  \quad \ \text{for some} \ \, m\ts.
$$
\end{thm}

\smallskip

In fact, the proof of Theorem~\ref{t:main-Sta} shows that, if \. $\ESta_k \ts \in \Sigmap_m$ \. for some $m$, then $\PH=\Sigmap_{m+1}$ (i.e. collapse to the $(m+1)$-th level).
In Section~\ref{s:AF}, we deduce Theorem~\ref{t:main-AF} from Theorem~\ref{t:main-Sta}.
For the proof, any fixed~$k$ in~\eqref{eq:Sta} suffices, of course.  In the opposite
direction, we prove the following extension of the Shenfeld and van Handel's result
mentioned above:

\smallskip

\begin{thm}\label{t:ESta-1} \.
$\ESta_1 \ts \in \ts \P$.
\end{thm}

\smallskip

Together, Theorems~\ref{t:main-Sta} and~\ref{t:ESta-1} complete the analysis of
equality cases of Stanley's inequality. 


\smallskip

\subsection{Combinatorial interpretation} \label{ss:intro-combin}
The problem of finding a \defna{combinatorial interpretation} \ts is
fundamental in both enumerative and algebraic combinatorics, and was
the original motivation of this investigation (see $\S$\ref{ss:hist-combin-int}).
Although very different in appearance and technical details, there are
certain natural parallels with the stability problems discussed above.

\smallskip

Let \ts $f\geqslant g$ \ts be an inequality between two counting
functions \ts $f,g \in \SP$.  We say that \ts $(f-g)$ \ts has a
\defn{combinatorial interpretation}, if \ts $(f-g) \in \SP$.
While many combinatorial inequalities have a combinatorial interpretation,
for the Stanley inequality~\eqref{eq:Sta} this is an open problem.  Formally,
let
\begin{equation*}
\Phi_{\bz\bc}(P, x,a) \, := \, \aN_{\bz\bc}(P, x,a)^2 \. -
\. \aN_{\bz\bc}(P, x,a+1) \cdot  \aN_{\bz\bc}(P, x,a-1)
\end{equation*}
denote the defect in~\eqref{eq:Sta}.
Let \ts $\phi_k: \big(P, X^{k+1}, [n]^{k+1}\big) \to \nn$ \ts
be the function computing \ts
$\Phi_{\bz\bc}(P, x,a)$. 

\smallskip

\begin{cor} \label{c:main-Stanley-not-SP}
For all \ts $k\ge 2$, function \ts $\phi_k$ \ts does not have a
combinatorial interpretation unless the polynomial hierarchy collapses
to the second level:
$$\phi_k \in \SP  \ \ \Longrightarrow \ \ \PH=\Sigmap_2\ts.
$$
\end{cor}

\smallskip

To see some context behind this result, note that \. $\aN_{\bz\bc}(P, x,a)\in \SP$ \.
by definition, so \. $\phi_k \in \GapPP$, a class of nonnegative functions in \ts
$\GapP:=\SP-\SP$.  We currently know very few functions which are in \ts $\GapPP$ \ts
but not in~$\ts \SP$.  The examples include
\begin{equation}\label{eq:maltese} \tag{\text{$\circledast$}}
\big(\text{\#3SAT}(F) \ts - \ts 1 \big)^2, \quad
\big(\text{\#2SAT}(F) \ts - \ts \text{\#2SAT}(F')\big)^2 \quad \text{and} \quad
\big(e(P)-e(P')\big)^2,
\end{equation}
where \ts $F,F'$ \ts are CNF Boolean formulas and \ts $P,P'$ \ts are posets \cite{CP23,IP22}.
In other words, all three functions in \eqref{eq:maltese} \emph{do not} \ts have a combinatorial
interpretation (unless $\PH$ collapses).
The corollary provides the first \emph{natural} \ts example of a defect function
that is \ts $\GapPP$ \ts but not in~$\ts \SP$.

The case \ts $k=0$, whether \. $\phi_0 \in \SP$, is especially interesting
and remains a challenging open problem, see \cite[$\S$9.12]{CPP-effective}
and \cite[Conj.~6.3]{Pak-OPAC}.
The corollary suggests that Stanley's inequality~\eqref{eq:Sta} is unlikely
to have a direct combinatorial proof, see~$\S$\ref{ss:finrem-injective}.

To understand how the corollary follows from the proof of Theorem~\ref{t:main-Sta},
note that \ts $\phi_2 \in \SP$ \ts implies that there is a polynomial
certificate for the Stanley inequality being strict.
In other words, we have \ts $\ESta_2 \in \coNP$, giving the assumption in the theorem.

\smallskip

\subsection*{Structure of the paper}
We begin with definitions and notation in Section~\ref{s:def}, followed
by the lengthy background and literature review in Section~\ref{s:hist}
(see also~$\S$\ref{ss:finrem-hist}).  In the key
Section~\ref{s:roadmap}, we give proofs of Theorems~\ref{t:main-AF}
and~\ref{t:main-Sta}, followed by proofs of Corollaries~\ref{c:main-AF-stab}
and~\ref{c:main-Stanley-not-SP}.  These results are reduced to several
independent lemmas, which are proved one by one in Sections~\ref{s:AF}--\ref{s:verify}.
We prove Theorem~\ref{t:ESta-1} in Section~\ref{s:Sta1}.   This section
is independent of the previous sections (except for notation in~$\S$\ref{ss:main-MS}).
We conclude with extensive final remarks and open problems in Section~\ref{s:finrem}.

\medskip

\section{Definitions and notation}\label{s:def}

\subsection{General notation}\label{s:def-gen}
Let \ts $[n]=\{1,\ldots,n\}$, \ts $\nn=\{0,1,2,\ldots\}$ \ts and \ts $\nz=\{1,2,\ldots\}$.
For a subset \ts $S\subseteq X$ \ts and element \ts $x\in X$,
we write \ts $S+x:=S\cup \{x\}$ \ts and \ts $S-x:=S\sm\{x\}$.
For a sequence \ts $\ba =(a_1,\ldots,a_m)$, denote
\ts $|\ba| := a_1 + \ldots + a_m$\ts.  This sequence is \defn{log-concave}, if
\. $a_i^2\ge a_{i-1} \ts a_{i+1}$ \. for all \ts $1< i < m$.

\smallskip

\subsection{Mixed volumes} \label{ss:def-mixed}
Fix \ts $n \geq 1$.
For two sets \. $A, B \subset \Rb^n$ \. and constants \ts $\al,\be>0$, denote by
\[ \al A \. + \. \be B \, := \, \bigl\{ \ts \al\xb \. + \.  \be \yb  \, : \, \xb \in A,  \. \yb \in B  \ts \bigr\}
\]
the \defnb{Minkowski sum} of these sets.
For a  convex body \ts $\iK \subset \Rb^n$ \ts with affine dimension $d$, denote by \ts $\Vol_d(\iK)$ \ts the
volume of~$\iK$.  We drop the subscript when \ts $d=n$.

One of the basic result in convex geometry is \defn{Minkowski's theorem}, see e.g.~\cite[$\S$19.1]{BZ-book},
that the volume of convex bodies with affine dimension $d$ behaves as a homogeneous polynomial
of degree~$d$ with nonnegative coefficients:

\smallskip

\begin{thm}[{\rm Minkowski}{}]\label{thm:Minkowski}
	For all convex bodies \. $\iK_1, \ldots, \iK_r \subset \Rb^n$ \. and \. $\lambda_1,\ldots, \lambda_r > 0$,
	we have:
	\begin{equation}\label{eq:mixed volume definition}
		\Vol_d(\lambda_1 \iK_1+ \ldots + \lambda_r \iK_r) \ =  \ \sum_{1 \ts \le \ts i_1\ts ,\ts \ldots \ts , \ts i_d\ts \le \ts r} \. \iV\bigl(\iK_{i_1},\ldots, \iK_{i_d}\bigr) \, \lambda_{i_1} \ts\cdots\ts \lambda_{i_d}\.,
	\end{equation}
	where the functions \ts $\iV(\cdot)$ \ts are nonnegative and symmetric, and where $d$ is the affine dimension of \ts $\lambda_1 \iK_1+ \ldots + \lambda_r \iK_r$ \ts $($which does not depend on the choice of \ts $\lambda_1,\ldots, \lambda_r)$.
\end{thm}

\medskip

The coefficients \. $\iV(\iA_{i_1},\ldots, \iA_{i_d})$ \. are called \defnb{mixed volumes}
of \. $\iA_{i_1}, \ldots, \iA_{i_d}$\..  We refer to \cite{HW,Leicht,Schn} for an accessible
introduction to the subject.

\smallskip

\subsection{Posets} \label{ss:def-posets}
For a poset \ts $P=(X,\prec)$ \ts and a subset \ts $Y \ssu X$, denote
by \ts $P_Y=(Y,\prec)$ \ts a \defn{subposet} \ts of~$P$.  We use \ts
$(P-z)$ \ts to denote a subposet \ts $P_{X-z}$\ts, where \ts $z\in X$.
Element \ts $x\in X$ \ts is \defn{minimal} \ts in~$\ts P$, if there exists no element  \ts $y \in X-x$ \ts such that \ts $y \prec x$\ts.  Define \defn{maximal} \ts
elements similarly.
Denote by \ts $\min(P)$ \ts and \ts $\max(P)$ \ts the set
of minimal and maximal elements in~$P$, respectively.

In a poset \ts $P=(X,\prec)$, elements \ts $x,y\in X$ \ts are called
\defn{parallel} or \defn{incomparable} if \ts $x\not\prec y$ \ts
and \ts $y \not \prec x$.  We write \. $x\parallel y$ \. in this case.
\defn{Comparability graph} \ts is a graph on~$X$, with edges
\ts $(x,y)$ \ts where \ts $x\prec y$.
Element \ts $x\in X$ \ts is said to \defn{cover}
\ts $y\in X$, if \ts $y\prec x$ \ts and there are no elements \ts $z\in X$ \ts
such that \. $y\prec z \prec x$.

A \defn{chain} is a subset \ts $C\ssu X$ \ts of pairwise comparable elements.
The \defn{height} of poset \ts $P=(X,\prec)$ \ts is the maximum size of a chain.
An \defn{antichain} is a subset \ts $A\ssu X$ \ts of pairwise incomparable elements.
The \defn{width} of poset  \ts $P=(X,\prec)$ \ts is the size of the maximal antichain.

A \defn{dual poset} \ts is a poset \ts $P^\ast=(X,\prec^\ast)$, where
\ts $x\prec^\ast y$ \ts if and only if \ts $y \prec x$.
%
%
%
A \defn{disjoint sum} \ts $P+Q$ \ts of posets \ts $P=(X,\prec)$ \ts
and \ts $Q=(Y,\prec')$ \. is a poset \ts $(X\cup Y,\prec^\di)$,
where the relation $\prec^\di$ coincides with $\prec$ and $\prec'$ on
$X$~and~$Y$, and \. $x\.\|\. y$ \. for all \ts $x\in X$, $y\in Y$.
A \defn{linear sum} \ts $P\oplus Q$ \ts of posets \ts $P=(X,\prec)$ \ts
and \ts $Q=(Y,\prec')$ \. is a poset \ts $(X\cup Y,\prec^\di)$,
where the relation $\prec^\di$ coincides with $\prec$ and $\prec'$ on
$X$~and~$Y$, and \. $x\prec^\di y$ \. for all \ts $x\in X$, $y\in Y$.

Posets constructed from one-element posets by recursively taking
disjoint and linear sums are called \defn{series-parallel}.
Both \defn{$n$-chain} \ts $C_n$ \ts and \defn{$n$-antichain} \ts $A_n$
\ts are examples of series-parallel posets. \defn{Forest} \ts is
a series-parallel poset formed by recursively taking disjoint sums
(as before), and linear sums with one element: \ts $C_1 \oplus P$.
We refer to \cite[Ch.~3]{Sta-EC} for an accessible introduction,
and to surveys \cite{BrW,Tro} for further definitions and standard
results.
\smallskip

\subsection{Poset polytopes}\label{ss:hist-poset-polytopes}
Let \ts $P=(X,\prec)$ \ts be a poset with \ts $|X|=n$ \ts elements.
The \defn{order polytope} \ts $\cO_P\subset \rr^n$ \ts is defined as
\begin{equation}\label{eq:order-def}
0\le \al_x \le 1 \quad \text{for all}  \quad x\in X\., \qquad
\al_x \le \al_y \quad \text{for all}  \quad  x\prec y, \ \ x,y \in X.
\end{equation}
Similarly, the \defn{chain polytope} (also known as the \defn{stable set polytope})
 \ts $\cS_P\subset  \rr^n$ \ts  is defined as
\begin{equation}\label{eq:chains-def}
\be_x \ge 0 \quad \text{for all}  \quad x\in X\., \qquad
\be_x + \be_y + \ldots \le 1 \quad \text{for all}  \quad  x\prec y \prec \cdots\,, \ x,y,\ldots \in X.
\end{equation}
In \cite{Sta-two}, Stanley computed the volume of both polytopes:
\begin{equation}\label{eq:two-poset}
\Vol \ts (\cO_P) \, = \, \Vol \ts (\cS_P) \, = \, \frac{e(P)}{n!}\..
\end{equation}
This connection is the key to many applications of geometry to
poset theory and vice versa.

\smallskip

\subsection{Terminology}\label{ss:def-term}
For functions \ts $f,g: X\to \rr$, we write \ts
\defn{$f\geqslant g$}, if \ts $f(x) \ge g(x)$ \ts for all \ts $x\in X$.
For an inequality \. $f \geqslant g$, 
the \defn{defect} \ts is a function \ts $h:=f-g$.
The \defn{equality cases} \ts to describe the set of
\ts $x\in X$ \ts such that \ts $f(x)=g(x)$.
Denote by \ts $X_h :=\{x\in X \. : h(x)=0\} \subseteq X$ \ts the subset
of equality cases.

We use \. {\sc E}$_h$ \. to denote the
\defn{equality verification} \ts of \. $f(x)=g(x)$, i.e.\
the decision problem
$$
\text{\sc E}_h \, := \, \big\{ \ts f(x) \. =^? \. g(x) \big\},
$$
where \ts $x\in X$ \ts is an input.  Since \. E$_h=\big\{x\in^? X_h\}$, this is
a special case of the
\defn{inclusion problem}.
We use \. {\sc V}$_h$ \. to denote
the \defn{verification} \ts of \. $h(x)=a$, i.e.\
the decision problem
$$
\text{\sc V}_h \, := \, \big\{ \ts f(x) \ts - \ts g(x) \. =^? \. a \big\},
$$
where \ts $a\in \rr$ \ts and \ts $x\in X$ \ts are the input.  Clearly,
{\sc V}$_h$ \ts is a more general problem than \ts {\sc E}$_h$\ts.

For a subset \ts $Y\subseteq X$, we use \defn{description} \ts for an
equivalent condition for the inclusion problem \ts $\big\{x\in^? Y\big\}$,
where \ts $x\in X$.
We use \defn{equality conditions} \ts for a description of \ts E$_h$.
We say that equality cases of \. $f\geqslant g$ \ts have a \ts \defn{description
in the polynomial hierarchy} \ts
if \. {\sc E}$_h \in \PH$. In other words, there
is a CNF Boolean formula \ts $\Phi(y_1,y_2,y_3,\ldots,x)$,  such that
$$
\forall \ts x\in X \ : \
\text{\sc E}_h \ \Longleftrightarrow \
\exists y_1 \. \forall y_2 \. \exists y_3  \. \ldots \. \Phi(y_1,y_2,y_3,\ldots,x).
$$

\smallskip

\subsection{Complexity} \label{ss:def-CS}
We assume that the reader is familiar with basic notions and results in
computational complexity and only recall a few definitions.  We use standard
complexity classes: \. $\P$, \. $\FP$, \. $\NP$,\. $\coNP$, \. $\SP$, \. $\Sigmap_m$ \. and \. $\PH$.
The notation \. $\{a =^? b\}$ \. is used to denote the
decision problem whether \ts $a=b$.  We use the \emph{oracle notation} \ts
{\sf R}$^{\text{\sf S}}$ \ts for two complexity classes \ts {\sf R}, {\sf S} $\subseteq \PH$,
and the polynomial closure \ts $\<${\sc A}$\>$ for a problem \ts {\sc A} $\in \PSPACE$.
We will also use less common classes \.
$$
\GapP:= \{f-g \mid f,g\in \SP\} \quad \text{and} \quad
\CEP:=\{f(x)=^?g(y) \mid f,g\in \SP\}.
$$
Note that \ts $\coNP \subseteq \CEP$.

We also assume that the reader is familiar with standard decision and
counting problems: \ts {\sc 3SAT}, \ts {\sc \#3SAT} \ts and
\ts {\sc PERMANENT}.  Denote by \ts {\sc \#LE} \ts the problem of
computing the number \ts $e(P)$ \ts of linear extensions.
%
For a counting function \ts $f\in \SP$,
the \defn{coincidence problem} \ts is defined as:
$$\text{\sc C}_f \ := \ \big\{\ts f(x) \. = ^? \ts f(y) \ts \big\}.
$$
Note the difference with the equality verification problem \ts E$_{f-g}$ \ts
defined above.
Clearly, we have both \ts $\text{\sc E}_{f-g}\in \CEP$ \. and \. $\text{\sc C}_f \in \CEP$.
Note also that \.$\text{\sc C}_\text{\#3SAT}$ \ts
is both \ts $\CEP$-complete \ts and \ts $\coNP$-hard.

The distinction between \emph{binary} \ts and \emph{unary} \ts presentation
will also be important.  We refer to \cite{GJ78} and \cite[$\S$4.2]{GJ79}
for the corresponding notions of $\NP$-completeness and \emph{strong} \ts $\NP$-completeness.
Unless stated otherwise, we use the word ``\emph{reduction}'' \ts to mean
``polynomial Turing reduction''.
We refer to \cite{AB,Gold,Pap} for definitions and standard results
in computational complexity.


\medskip

\section{Background and historical overview}\label{s:hist}

\subsection{Geometric inequalities} \label{ss:hist-geom}
The history of equality conditions of geometric inequalities goes back to
antiquity, see e.g.\ \cite{Bla,Porter}, when it was discovered that the
\defn{isoperimetric inequality}
\begin{equation}\label{eq:Isop}\tag{Isop}
\ell(X)^2 \, \ge \, 4 \ts \pi \ts a(X)
\end{equation}
is an equality \. \underline{if and only if} \. $X$ \ts is a circle.
Here \ts $\ell(X)$ \ts is the perimeter and \ts $a(X)$ \ts is the area of a convex \ts
$X\ssu \rr^2$.  This classical result led to numerous extensions and generalizations
leading to the Alexandrov--Fenchel inequality~\eqref{eq:AF}.
We refer to 
\cite{BZ-book,Schn} for a review of the literature.

Below we highlight only the most important developments to emphasize
how the equality conditions become more involved as one moves in the
direction of the AF~inequality
(see also~$\S$\ref{ss:finrem-discrete-isop} and $\S$\ref{ss:finrem-BM}).
The celebrated \defn{Brunn--Minkowski inequality} \ts states that for all
convex \ts $\iK, \iL \subset \rr^d$, we have:
\begin{equation}\label{eq:BM}\tag{BM}
\Vol(\iK+\iL)^{1/d} \, \ge \,     \Vol(\iK)^{1/d} \. + \. \Vol(\iL)^{1/d},
\end{equation}
see e.g.\ \cite{Gar02} for a detailed survey.  This inequality
``plays an important role in almost all branches of mathematics'' \cite{Bar07}.
Notably, both Brunn and Minkowski showed the equality in~\eqref{eq:BM}
holds \. \underline{if and only if} \. $\iK$ \ts is an expansion of~$\ts \iL$.

For the \defn{mean width inequality}
\begin{equation}\label{eq:Mink-mean} \tag{MWI}
s(\iK)^2 \, \ge \, 6\ts \pi \. w(\iK) \. \Vol(\iK)\ts,
\end{equation}
for all convex $\ts \iK \subset  \rr^3$,
Minkowski conjectured (1903) the equality cases are the \emph{cap bodies}
(balls with attached tangent cones). Here \ts $s(\iK)$ \ts
is the surface area and \ts $w(\iK)$ \ts is the
\emph{mean width} \ts of $\ts \iK$.  Minkowski's conjecture
was proved by Bol (1943), see e.g.\ \cite{BF,BZ-book}.

The \defn{Minkowski's quadratic inequality} \ts for three convex
bodies \ts $\iK,\iL,\iM \subset  \rr^3$, states:
\begin{equation}\label{eq:MQI} \tag{MQI}
\iV(\iK,\iL,\iM)^2 \, \ge \, \iV(\iK,\iK,\iM) \. \cdot \. \iV(\iL,\iL,\iM)\ts.
\end{equation}
This is a special case of~\eqref{eq:AF}
for \ts $n=d=3$.  When \ts $\iL=\iB_1$ \ts is a unit
ball and \ts $\iK=\iM$, this gives~\eqref{eq:Mink-mean}.
Favard~\cite[p.~248]{Fav} wrote that the equality
conditions for \eqref{eq:MQI} \emph{``parait difficile \`a \'enonce''}
(``seem difficult to state'').  There are even interesting families of
convex polytopes that give equality cases (see e.g.\ \cite[Fig. 2.1]{SvH-acta}).

Shenfeld and van Handel \cite{SvH-duke} gave a complete characterization
of the equality cases of \eqref{eq:MQI} as triples of convex bodies
that are similarly truncated in a certain formal sense.  Notably, for
the full-dimensional H-polytopes in~$\rr^3$, each with at most
\ts $n$ \ts facets, the equality conditions amount to checking
\ts $O(n)$ \ts linear relations for distances between facet
inequalities.  This can be easily done in polynomial time.

\smallskip

\subsection{Alexandrov--Fenchel inequality} \label{ss:hist-AF}
For the AF inequality~\eqref{eq:AF}, the equality conditions have long
believed to be out of reach, as they would generalize those for
\eqref{eq:Mink-mean} and~\eqref{eq:MQI}.  Alexandrov made a
point of this in his original 1937 paper:

\smallskip

\begin{center}\begin{minipage}{12.8cm}%
{{\em ``Serious difficulties occur in determining the conditions for equality to
hold in the general inequalities just derived''} \cite[$\S$4]{Ale37}.
}
\end{minipage}\end{center}

\smallskip

\nin
Half a century later, Burago and Zalgaller reviewed the literature and summarized:

\smallskip

\begin{center}\begin{minipage}{12.8cm}%
{{\em ``A conclusive study of all these situations when the equality sign holds has not
been carried out, probably because they are too numerous''} \cite[$\S$20.5]{BZ-book}.
}
\end{minipage}\end{center}

\smallskip

\nin
Schneider made a case for perseverance:

\smallskip

\begin{center}\begin{minipage}{12.8cm}%
{{\em ``As \eqref{eq:AF} represents a classical inequality of fundamental importance
and with many applications, the identification of the equality cases is a problem of
intrinsic geometric interest.  Without its solution, the Brunn--Minkowski theory of
mixed volumes remains in an uncompleted state.''} \cite[p.~426]{Schn94}.
}
\end{minipage}\end{center}

\smallskip

The AF inequality has a number of proofs using ideas from convex geometry,
analysis and algebraic geometry, going back to two proofs by Alexandrov
(Fenchel's full proof never appeared).
We refer to \cite{BZ-book,Schn} for an overview of the older literature,
especially \cite[p.~398]{Schn} for historical remarks,  and to
\cite{BL23,CP2,CEKMS,KK,SvH-pams,Wang} for some notable recent proofs.
%
%
All these proofs use a limit argument at the end, which can create new equality
cases that do not hold for generic convex bodies.
This partially explains the difficulty of the problem
(cf.~$\S$\ref{ss:finrem-equality} and \cite[Rem.~3.7]{SvH-duke}).

In~\cite{Ale37}, Alexandrov gave a description of equality cases for
combinatorially isomorphic polytopes.  This is a large family of
full-dimensional polytopes, for which every convex body is a limit.
In particular, he showed that for the full-dimensional axis-parallel boxes
\ts $[\ell_1\times \ldots \times \ell_n]$, the equality in~\eqref{eq:AF} is equivalent
to $\iK$ and $\iL$ being homothetic (cf.~$\S$\ref{ss:finrem-vdW}).

In the pioneering work~\cite{Schn85}, Schneider published a conjectural description
of the equality cases, corrected later by Ewald~\cite{Ewald}, see also~\cite{Schn}.
After many developments, this conjecture was confirmed for all \.
smooth (full-dimensional) convex bodies \ts $\iQ_i$ \ts \cite{Schn90a}, and for all
(not necessarily full-dimensional)  convex bodies \ts $\Q_1=\ldots=\iQ_{n-2}$\ts,
by Shenfeld and van Handel \cite{SvH-acta}.  Closer to the subject
of this paper, in a remarkable development, the authors gave
a geometric description of the equality cases \emph{for all} \ts
convex polytopes.  They explain:

\smallskip

\begin{center}\begin{minipage}{12.8cm}%
{{\em ``Far from being esoteric, it is precisely the case of convex bodies with
empty interior $($which is not covered by previous conjectures$)$ that arises
in combinatorial applications''} \cite[$\S$1.3]{SvH-acta}.
}
\end{minipage}\end{center}

The geometric description of the equality cases in \cite{SvH-acta} is indirect,
technically difficult to prove, and computationally hard in the degenerate
cases.\footnote{It follows from our Theorem~\ref{t:main-AF}
that it has to be, see a discussion in~$\S$\ref{ss:finrem-meaning}.}  While we will
not quote the full statement (Theorem~2.13 in \cite{SvH-acta}), let us mention the
need to find witnesses polytopes \ts $\iM_i \ts, \iN_i \ssu \rr^n$ \ts which must
satisfy certain conditions (Def.~2.10, ibid.)  The second
of these conditions is an equality of certain mixed volumes (Eq.~(2.4), ibid.)

In \cite[$\S$2.2.3]{SvH-acta}, the authors write:  ``Condition (2.4) should be
viewed merely as a normalization''.\footnote{By that the authors of \cite{SvH-acta}
seem to mean that their description captured all the geometry in the problem,
as opposed to the equality of mixed volumes which has no geometric content. }
From the computational complexity point of view, asking for the equality of
mixed volumes (known to be hard to compute, see~$\S$\ref{ss:hist-CA}),
lifts the problem outside of the polynomial hierarchy,
to a hard coincidence problem (see~$\S$\ref{ss:def-CS}).
This coincidence problem eventually percolated into \cite{MS22},
see~\eqref{eq:MS} below, which in turn led directly to this work.

\smallskip

\subsection{Stability} \label{ss:hist-stab}
\defn{Bonnesen's inequality} \ts is an extension of the isoperimetric
inequality~\eqref{eq:Isop}, which states that for every convex \ts
$X\ssu \rr^2$, we have:
\begin{equation}\label{eq:Bon}\tag{Bon}
\ell(X)^2 \. - \. 4\ts \pi \ts a(X) \, \geq \, 4 \ts \pi \ts (R-r)^2,
\end{equation}
where \ts $R$ \ts is the smallest radius of the circumscribed circle,
and \ts $r$ \ts is the maximal radius of the inscribed circle.\footnote{Note
that when \ts $X \ssu \rr^2$ \ts is a nonzero interval, we have \ts $r=0$ \ts and \ts $\ell(X)=4R$,
so the inequality remains strict.}
Moreover, Bonnesen proved~\cite{Bon}, that there is an \defng{annulus}
({thin shell}) \ts $U$ \ts between concentric circles of radii \ts $R\ge r$,
such that \ts $\partial X\subseteq U$ \ts and~\eqref{eq:Bon} holds.
Note that the optimal such annulus can be computed in polynomial
time, see \cite{AAHS}.

Bonnesen's inequality~\eqref{eq:Bon} was an inspiration for many
Bonnesen type inequalities \cite{Oss78,Oss79,Gro90}.
See also discrete versions in~$\S$\ref{ss:finrem-discrete-isop},
and applications in computational geometry in~\cite{KS99}.
There is now an extensive literature on stability inequalities in geometric
and more general context, see e.g.\ \cite{Fig13,Gro93}.


There is an especially large literature on the stability of the
Brunn--Minkowski inequality~\eqref{eq:BM}.  
For major early advances
by Diskant (1973), Groemer (1988) and others, see e.g.\ \cite{Gro93}
and references therein.  We refer to~\cite{Fig14} for
an overview of more recent results, including \cite{FMP09,FMP10}.
See also \cite{EK14} for the thin shell \ts
type bounds, and~\cite{FJ17} for the stability of~\eqref{eq:BM} for
\emph{nonconvex sets}.


For the Alexandrov--Fenchel inequality~\eqref{eq:AF}, the are very few
stability results, all for the full dimensional convex bodies with various
regularity conditions, see e.g.\ \cite{Mar17,Schn90}.

\smallskip

\subsection{Linear extensions} \label{ss:hist-LE}
Linear extensions play a cental role in enumerative combinatorics
and order theory.  They appear in connection with saturated chains
in \emph{distributive lattices}, \emph{standard Young tableaux}
and \emph{$P$-partitions}, see e.g.\ \cite{Sta-EC}.

The world of inequalities for linear extensions has a number of
remarkable results, some with highly nontrivial equality conditions.
Notably, the \defng{Bj\"orner--Wachs inequality} \ts for \ts $e(P)$ \ts
is an equality if and only if \ts $P$ \ts is a forest
\cite[Thm~6.3]{BW89}, see also~\cite{CPP-effective}.
On the other hand, the celebrated \defng{XYZ inequality} \ts
established by Shepp in \cite{She-XYZ} (see also~\cite[$\S$6.4]{ASE}),
has no nontrivial equality cases \cite{Fish}.

An especially interesting example is the \defn{Sidorenko inequality} \ts
\begin{equation}\label{eq:Sid1}
e(P) \cdot e(P^\circ) \. \ge \. n!
\end{equation}
for posets \ts $P, P^\circ$ \ts on the same ground set with \ts $n$ \ts elements,
which have complementary comparability graphs \cite{Sid} \ts (other proofs are given
in~\cite{CPP-effective,GG22}).
Sidorenko also proved that the series-parallel posets are the only
equality cases.
This solves the equality verification problem of \eqref{eq:Sid1}, since the
recognition problem of series-parallel posets is in~$\P$, see~\cite{VTL}.

It was noticed in~\cite{BBS}, that the Sidorenko inequality follows from
\defng{Mahler's conjecture}, which states that for every convex centrally symmetric
body \ts $\iK\subset  \rr^n$, we have:
\begin{equation}\label{eq:Mah}
\Vol(\iK) \cdot \Vol(\iK^{o}) \, \ge \, \frac{4^n}{n!}\,.
\end{equation}
To derive \eqref{eq:Sid1} from \eqref{eq:Mah}, take \ts $\iK$ \ts to be the
union all axis reflections the chain polytope \ts $\cS_P$ \ts defined
in~\eqref{eq:chains-def}.  Mahler's
conjecture \eqref{eq:Mah} is known for all axis symmetric convex bodies
\cite{StR}, but remains open in full generality \cite{AASS},
in part due to the many equality cases \cite[$\S$1.3]{Tao}.

\smallskip

\subsection{Stanley inequality} \label{ss:hist-Stanley}
Stanley's inequality \eqref{eq:Sta} is of independent interest in order
theory, having inspired a large literature especially in the last few years.
The case \ts $k=0$ \ts is especially
interesting.  The unimodality in this case was independently conjectured by
Kislitsyn \cite{Kis68} and Rivest,  while the log-concavity was conjectured
by Chung, Fishburn and Graham \cite{CFG},
who established both conjectures for posets of width two.  Stanley proved them
in \cite{Sta-AF} in full generality.\footnote{For \ts $k \ge 1$, the
inequality~\eqref{eq:Sta} is sometimes called the \defng{generalized Stanley
inequality}, see \cite{CPP-effective}.}  The authors of \cite{CFG}
called Rivest's conjecture ``tantalizing'' and Stanley's proof ``very ingenious''.

The \defng{Kahn--Saks inequality} \ts is a generalization of the \ts $k=0$ \ts
case of \eqref{eq:Sta}, and is also proved from the AF inequality.
This inequality was used to obtain the first positive result in the direction
of the \. \defng{$\frac13-\frac23$ conjecture}~\cite{KS}.  For posets of
width two, both the \ts $k=0$ \ts  case of the Stanley inequality, and the Kahn--Saks
inequality have natural \emph{$q$-analogues} \cite{CPP-KS}.  A generalization
of Stanley's inequality to \emph{marked posets} \ts was given in~\cite{LMS19}.

For all \ts $k\ge  0$, the \defn{vanishing conditions} \ts
$\{\aN_{\bz\bc}(P, x,a)=^?0\}$ \ts are in \ts $\P$.  This was shown by
David and Jacqueline Daykin in \cite[Thm~8.2]{DD}, via explicit necessary
and sufficient conditions.  Recently, this result was rediscovered
in \cite[Thm~1.11]{CPP-effective} and \cite[Thm~5.3]{MS22}.
Similarly, the \defn{uniqueness conditions} \ts $\{\aN_{\bz\bc}(P, x,a)=^?1\}$ \ts
are in~$\poly$ \ts by the result of Panova and the authors \cite[Thm~7.5]{CPP-effective},
where we gave explicit necessary and sufficient conditions.
Both the vanishing and the uniqueness conditions give
examples of equality cases of the Stanley inequality, which remained
a ``major challenge'' in full generality \cite[$\S$9.10]{CPP-effective}.

As we mentioned in the introduction, Shenfeld and van Handel resolved
the \ts $k=0$ \ts case of Stanley equality conditions by giving
explicit necessary and sufficient conditions, 
which can be verified in polynomial time, see \cite{SvH-acta}.
Similar explicit necessary and sufficient conditions for the
Kahn--Saks inequality were conjectured in \cite[Conj.~8.7]{CPP-KS}, and
proved for posets of width two.  Building on the technology in
\cite{SvH-acta}, van Handel, Yan and Zeng gave the proof of this
conjecture in~\cite{vHY}.

In \cite{CP}, we gave a new proof of the \ts $k=0$ \ts
case of \eqref{eq:Sta}, using a \defng{combinatorial atlas} \ts
technology.  This is an inductive self-contained linear algebraic
approach; see \cite{CP2} for the introduction.
We also gave a new proof of the Shenfeld and van Handel equality conditions,
and generalized both results to \defng{weighted linear extensions}
(see $\S\S$1.16-18 in \cite{CP}).

In an important development, Ma and Shenfeld \cite{MS22} advanced the
technology of \cite{SvH-acta}, to give a clean albeit ineffective
combinatorial description of the equality cases in full generality.
In particular, they showed that \eqref{eq:AF} is an equality \ts \underline{if and only if}
\begin{equation}\label{eq:MS}
\aN_{\bz\bc}(P, x,a-1)\, = \, \aN_{\bz\bc}(P, x,a)\, = \,  \aN_{\bz\bc}(P, x,a+1).
\end{equation}
They proceeded to give explicit necessary and sufficient conditions
for these equalities in some cases (see~$\S$\ref{ss:finrem-eq}).
About the remaining cases that they called \defng{critical} (see~$\S$\ref{ss:Sta1-critical}),
they write: ``It is an interesting problem to find
\ts $[$an explicit description$]$ \ts for critical posets'' \cite[Rem.~1.6]{MS22}.
Our Theorem~\ref{t:main-Sta} implies that such a description is unlikely, as it
would imply a disproof of a major conjecture in computational complexity
(see also~$\S$\ref{ss:finrem-meaning}).


\smallskip

\subsection{Complexity aspects} \label{ss:hist-CC}
There are two standard presentations of polytopes: \defng{H-polytopes} \ts described by the
inequalities and \defng{V-polytopes} \ts described by the vertices.  These two presentation
types have very different nature in higher dimensions, see e.g.\ \cite{DGH}.
We refer to \cite{GK94,GK97} for an overview of standard complexity problems in
geometry, and to \cite[$\S$19]{Schr86}, \ts \cite[$\S$5.16]{Schr03},
for the background on totally unimodular matrices and $\TU$-polytopes.
Note also that testing whether matrix \ts $A$ \ts is totally unimodular
can be done in polynomial time, see \cite{Sey}.

When the dimension $n$ is bounded, H-polytopes and V-polytopes have the
same complexity, so the volume and the mixed volumes are in \ts $\FP$.
Thus, the dimension~$n$ is unbounded throughout the paper.
The volume of TU-polytopes is $\SP$-hard via reduction to \ts {\sc KNAPSACK} \cite{DF}.
Note that for \emph{rational} \ts H-polytopes in \ts $\rr^n$, the volume denominators
can be doubly exponential \cite{Law91}, thus not in \ts $\PSPACE$.
This is why we constrain ourselves to TU-polytopes which is a subclass
of H-polytopes that includes all order polytopes (see~$\S$\ref{ss:AF-slices}).

The mixed volume \ts  $\iV(\iQ_1,\ldots,\iQ_n)$ \ts coincides with the permanent when all \ts $\iQ_i$ \ts are
axis parallel boxes, see~\cite{vL} and~$\S$\ref{ss:finrem-vdW}.  Thus, computing
the mixed volume is \ts $\SP$-hard even for the boxes, see \cite{DGH}.
For rational H-polytopes, the vanishing problem \ts $\{\iV(\cdot)=^?0\}$ \ts
can be described combinatorially, and is thus in \ts $\NP$, see \cite{DGH,Est10}.
It is equivalent to computing the rank of intersection of two
geometric matroids (with a given realization), which is in~$\P$,
see \cite[$\S41$]{Schr03}.  For TU-polytopes in~$\rr^n$, the uniqueness problem
\ts $\big\{\iV(\cdot)=^?\frac{1}{n!}\big\}$ \ts is in \ts $\NP$ \ts by a result in~\cite{EG15}.

The problem \ts {\sc \#LE} \ts is proved \ts $\SP$-complete by Brightwell and Winkler
\cite[Thm~1]{BW},
and this holds even for posets of height two~\cite{DP}.
Linial noticed~\cite{Lin}, that this result and~\eqref{eq:two-poset}
together imply that the volume of H-polytopes is \ts $\SP$-hard even
when the input is in unary.  Linial also observed that the number
of vertices of order polytopes is $\SP$-complete (ibid.)

Now, fix \ts $k\ge 0$, \ts $x\in X$ \ts and \ts $\bz\in X^k$.  Clearly, we have:
$$
e(P) \ = \ \sum_{a\in [n]} \, \sum_{\bc \in [n]^k} \, \aN_{\bz\bc}(P, x,a)\ts,
$$
where the summation has size \ts $O(n^{k+1})$.
Thus, computing \. $\aN_{\bz\bc}(P, x,a)$ \. is also \ts $\SP$-complete.

Finally, it was proved in \cite{CP23}, that \ts $\text{\sc C}_\text{\#3SAT}$,
\ts $\text{\sc C}_\text{PERMANENT}$ \ts and \ts
$\text{\sc C}_\text{\#LE}$ \ts are not in \ts $\PH$, unless
\ts $\PH$ \ts collapses to a finite level.  The proof idea of
Theorem~\ref{t:main-Sta} is inspired by these results.


\smallskip

\subsection{Combinatorial interpretations} \label{ss:hist-combin-int}
Finding a combinatorial interpretation is a standard problem throughout
combinatorics, whenever a positivity phenomenon or an inequality
emerges.  Having a combinatorial interpretation allows one to
deeper understand the underlying structures, give asymptotic and
numerical estimates, as well as analyze certain algorithms.
We refer to \cite{Huh,Sta-log-concave,Sta-pos} for an overview of inequalities
in algebraic combinatorics and matroid theory, and to \cite{Pak-OPAC}
for a recent survey from the complexity point of view.

Recall that \ts $\GapP:=\SP-\SP$ \ts is a the class of difference
of two \ts $\SP$ \ts functions, and let \ts $\GapPP$ \ts be a subclass
of \ts $\GapP$ \ts of nonnegative functions.  Thus, for every
inequality \ts $f\geqslant g$ \ts of counting functions \ts $f,g\in \SP$,
we have \ts $(f-g) \in \GapPP$.  It was shown in \cite[Prop.~2.3.1]{IP22},
that \ts $\GapPP\ne \SP$, unless \ts $\PH = \Sigmap_2$.  The key example
is
{\small
$$\big(\text{\#3SAT}(F) \ts - \ts \text{\#3SAT}(F')\big)^2,
$$}
\hskip-0.1cm
see also the first function in \eqref{eq:maltese}.  The other two
functions in \eqref{eq:maltese} were given in~\cite{CP23}.
A \emph{natural} \ts $\GapPP$ \ts problem of computing $S_n$ character squared:
\ts $[\chi^\la(\mu)]^2$, was proved not in~$\ts \SP$ \ts (in unary),
under the same assumptions \cite{IPP22}.

The idea that some natural combinatorial inequalities can have no
combinatorial interpretations appeared in~\cite{Pak}. A number of interesting
examples were given in~\cite[$\S$7]{IP22}, including the
\emph{Cauchy}, \emph{Minkowski}, \emph{Hadamard}, \emph{Karamata}
and \emph{Ahlswede--Daykin inequalities}, all proved not in \ts $\SP$
\ts under varying complexity assumptions.

Closer to the subject of this paper, Ikenmeyer and the second author showed
that the AF defect \ts $\de(\cdot)$ \ts is not in~$\ts\SP$
(unless $\PH=\Sigmap_2$),
even for axis parallel rectangles in $\rr^2$ whose edge length are given
by \ts {\sc \#3SAT} \ts formulas \cite[Thm~7.1.5]{IP22}.
This is a nonstandard model of computation.  One can think of our Main
Theorem~\ref{t:main-AF} as a tradeoff: in exchange for needing a
higher dimension, we now have unary input and the standard model of computation.

\smallskip

\subsection{Complexity assumptions} \label{ss:hist-CA}
The results in the paper use different complexity assumptions,
and navigating between them can be confusing.  Here is short list of standard implications:
$$
\PH \ne \Sigmap_m \ \ \text{for all} \ \ m\ge 2 \quad \Longrightarrow \quad \PH \ne \Sigmap_2 \quad \Longrightarrow \quad \PH \ne \NP
\quad \Longrightarrow \quad
\P \ne \NP\ts.
$$
In other words, the assumption in Theorems~\ref{t:main-AF} and~\ref{t:main-Sta}
is the strongest, while \ts $\P\ne \NP$ \ts is the weakest.  Proving either
of these would be a major breakthrough in theoretical computer science.
Disproving either of these would bring revolutionary changes to the
way the computational complexity understands the nature of computation.
We refer to \cite{Aar16,Wig19} for an extensive discussion, philosophy
and implications in mathematics and beyond.

\medskip



\section{Proof roadmap} \label{s:roadmap}
The results in the paper follow from a series of largely independent
polynomial reductions and several known results.  In this section,
we only state the reductions whose proofs will be given in the
next few sections.  We then deduce both theorems from these reductions.

\subsection{Around Stanley equality} \label{ss:roadmap-lemmas}
First, we show that Theorem~\ref{t:main-AF} follows from Theorem~\ref{t:main-Sta}.
Recall the notation from the introduction.  Let \ts $P=(X,\prec)$ \ts be a poset
on \ts $|X|=n$ \ts elements.  As before, let \ts $x\in X$, \ts $a\in [n]$, \ts $\zb \in X^k$,
and \ts $\bc \in [n]^k$.  Recall also
{\small
\begin{equation*}
\EAF \, := \,
\big\{ 	\iV\big(\iK,\iL,\iQ_1,\ldots,\iQ_{n-2}\big)^2  \, =^? \, \iV\big(\iK,\iK,\iQ_1,\ldots,\iQ_{n-2}\big)
\cdot    \iV\big(\iL,\iL,\iQ_1,\ldots,\iQ_{n-2}\big) \big\},
\end{equation*}
%
%
\begin{equation*}
\ESta_k \, := \,
	\big\{	
\aN_{\bz\bc}(P, x,a)^2 \, =^? \, \aN_{\bz\bc}(P, x,a+1) \cdot   \aN_{\bz\bc}(P, x,a-1)
 \big\},
\end{equation*}
}
where \ts $\aN_{\zb \cb}(P,x,a)$ \ts are defined in~$\S$\ref{ss:intro-Stanley}.

\smallskip

\begin{prop}[{\rm cf.\ \cite[$\S$3]{Sta-AF}}{}]
 \label{p:AF-Sta}
For all \ts $k\ge 0$, \.
{\sc EqualityStanley}$_k$ \. reduces to \. {\sc EqualityAF}.
\end{prop}

\smallskip

The proof of the proposition is given in Section~\ref{s:AF}, is very
close to Stanley's original proof of the inequality~\eqref{eq:Sta}.
The key difference is the observation that slices of order polytopes
are TU-polytopes.  Next, we need a simple technical result.

\smallskip

\begin{lemma} \label{l:Sta-Sta-more}
For all \ts $k> \ell$, \.
{\sc EqualityStanley$_\ell$} \. reduces to \. {\sc EqualityStanley$_k$}\..
\end{lemma}

\begin{proof}
Let \ts $P=(X,\prec)$ \ts be a poset on $n$ elements, and let \ts $\bz\in X^k$,
\ts $\bc\in [n]^k$, \ts $x\in X$, \ts $a\in [n]$ \ts be as in~$\S$\ref{ss:intro-Stanley}.
Denote by \ts $P':=P+A_{k-\ell}$ \ts a poset obtained by adding \ts
$(k-\ell)$ \ts independent elements $z'_1,\ldots,z_{k-\ell}'$\ts.
Let \ts $c_i':=n+i$, for all \ts $1\le i\le k-\ell$.
For \ts $\bz':=\big(z_1,\ldots,z_\ell,z_1',\ldots,z_{k-\ell}'\big)$ \ts and
\ts $\bc':=\big(c_1,\ldots,c_\ell,c_1',\ldots,c_{k-\ell}'\big)$, we have:
$$\aN_{\bz'\bc'}(P', x,a) \, = \, \aN_{\bz\bc}(P, x,a).
$$
Varying $a$, we conclude that \. {\sc EqualityStanley$_k$} \. is equivalent to
\. {\sc EqualityStanley$_\ell$} \. in this special case.  This gives the
desired reduction.
\end{proof}

\smallskip

Next, we simplify the Stanley equality problem to
the following \ts \defn{flatness problem}:
$$
\text{\sc FlatLE}_k \ := \ \big\{\aN_{\zb \cb}(P,x,a) =^? \aN_{\zb \cb}(P,x,a+1)\big\}.
$$
The idea is to ask whether \ts $a$ \ts is in the flat part of the distribution
of \ts $f(x)$ \ts (cf.\ Figure~15.1 in~\cite{SvH-acta}).


\smallskip

\begin{lemma} \label{l:Flat-Sta}
For all \ts $k \geq 0$, \.
{\sc FlatLE$_k$} \. reduces to \. {\sc EqualityStanley$_{k+2}$}\..
\end{lemma}


We prove Lemma~\ref{l:Flat-Sta} in Section~\ref{s:Flat-Sta}.

\smallskip
\subsection{Relative numbers of linear extensions} \label{ss:roadmap-relative}
Let \ts $P=(X,\prec)$ \ts be a poset on \ts $|X|=n$ \ts elements,
and let \ts $\min(P)\subseteq X$ \ts be
the set of minimal elements of~$P$.  For every \ts $x \in \min(P)$, define
the \defn{relative number of linear extensions}:
\begin{equation}\label{eq:rho-def}
\ag(P,x) \, := \, \frac{e(P)}{e(P-x)}\..
\end{equation}
In other words, \. $\ag(P,x) = \pp[f(x)=1]^{-1}$,
where \. $f\in \Ec(P)$ \. is a uniform random linear extension of~$P$.
Denote by \. {\sc \#RLE} \. the problem of computing \ts $\ag(P,x)$.

\smallskip

\begin{lemma} \label{l:RLE}
	$\text{\sc \#RLE}$ \.
	is polynomial time equivalent to \. {\sc \#LE}.
\end{lemma}

\begin{proof}
By definition, $\text{\sc \#RLE}$ \ts reduces to \ts $\text{\sc \#LE}$.
In the opposite direction, let \ts $P=(X,\prec)$ \ts be a poset on \ts $|X|=n$ \ts elements.
Fix a linear extension \ts $g\in \Ec(P)$, and let \ts $x_i:=g^{-1}(i)$, \ts
$1\le i \le n$. Denote by \ts $P_i$ \ts a subposet of~$P$ restricted to \ts
$x_i,\ldots,x_n$ \ts and observe that \ts $x_i \in \min(P_i)$.  We have:
$$
e(P) \ = \ \frac{e(P_1)}{e(P_2)} \. \cdot \. \frac{e(P_2)}{e(P_3)} \, \cdots \
=  \ \ag(P_1,x_1) \. \cdot \. \ag(P_2,x_2) \, \cdots \,,
$$
which gives the desired reduction from \ts $\text{\sc \#LE}$ \. to \ts $\text{\sc \#RLE}$.
\end{proof}

\smallskip


We relate {\sc RLE} to flatness equality through the following series of reductions.
Consider the following coincidence problem:
\begin{equation*} 
	\text{\sc CRLE} \ := \ \big\{ \.	
	\rho(P,x) \, =^? \, \rho(Q,y)
	\. \big\},
\end{equation*}
where \ts $P=(X,\prec)$, \ts $Q=(Y,\prec')$ \ts are posets, and \ts $x\in \min(P)$,  \ts $y \in \min(Q)$\ts.

\smallskip

\begin{lemma}[{\rm see Theorem~\ref{thm:CRLE-Flat}}{}]\label{lem:CRLE-Flat}
	$\text{\sc CRLE}$ \.
	reduces to \. $\text{\sc FlatLE}_{0}$\..
\end{lemma}

\smallskip

Next, consider the following decision problem:
\begin{equation*}
	\text{\sc QuadRLE} \ := \ \big\{ \.	
	\ag(P_1,x_1) \cdot \ag(P_2,x_2) \, =^? \,
	\ag(P_3,x_3) \cdot \ag(P_4,x_4)
	\. \big\},
\end{equation*}
where \ts $P_1,P_2,P_3,P_4$ \ts are finite posets and \ts $x_i\in \min(P_i)$, for all \. $1\le i \le 4$.

\smallskip

\begin{lemma}[{\rm see Theorem~\ref{thm:gCF}}{}]\label{l:Quad-CRLE}
	$\text{\sc QuadRLE}$ \.
	reduces to \. $\text{\sc CRLE}$\ts.
\end{lemma}

\smallskip


\subsection{Verification lemma} \label{ss:roadmap-TCS}
Let \ts $P=(X,\prec)$ \ts be a poset on \ts $|X|=n$ \ts elements, and let \ts $x\in \min(P)$.
Consider
\begin{equation}\label{eq:VRLE}  
\text{\sc VerRLE} \ := \	\Big\{ \.	
	\ag(P,x) \.  =^? \,
	\frac{A}{B}
	\. \Big\},
\end{equation}
where $\ts A,B \ts$ are coprime integers with \ts $1\le B \le A\leq n!$ \.
We need the following:

\smallskip

\begin{lemma}[{\rm Verification lemma}{}]
\label{l:Verify-Quart}
\.  $\NP^{\<\text{\sc VerRLE}\>} \. \subseteq \. \NP^{\<\text{\sc QuadRLE}\>}.$
\end{lemma}

Note that the opposite direction \. ``$\supseteq$'' \. is also true and easy to prove.
Indeed, suppose you have an oracle \. \text{\sc VerRLE}.
Guess the values  \. $a_i:=\ag(P_i,x_i)\in \qqq$, verify that they are correct,
and check that \. $a_1\cdot a_2 = a_3\cdot a_4\ts.$  This gives \. \text{\sc QuadRLE}.
We will only need the direction in the lemma which is highly nontrivial.

\smallskip

\subsection{Putting everything together} \label{ss:roadmap-proofs}
We can now obtain all the results stated in the introduction, except for
Theorem~\ref{t:ESta-1} which uses different tools and is postponed until
Section~\ref{s:Sta1}.

\smallskip

\begin{proof}[Proof of Theorem~\ref{t:main-Sta}]
Recall that \ts {\sc \#LE} \ts is \ts $\SP$-complete \cite{BW} (see also~$\S$\ref{ss:hist-CC}).
By Lemma~\ref{l:RLE}, we conclude that \ts {\sc \#RLE} \ts is \ts $\SP$-hard.  We then have:
\begin{equation}\label{eq:together-LE}
	\ComCla{PH} \. \subseteq \. \P^{\SP} \. \subseteq \. \P^{\<\text{\sc \#RLE}\>} \. \subseteq \. \NP^{\<\text{\sc VerRLE}\>}
\end{equation}
where the first inclusion is Toda's theorem \cite{Toda}, the second inclusion is because \ts {\sc \#RLE} \ts
is $\SP$-hard, and the third inclusion is because one can simulate \ts {\sc \#RLE} \ts by first
guessing and then verifying the answer.

Fix \ts $k\ge 2$.  Combining Lemmas~\ref{l:Sta-Sta-more}, \ts \ref{l:Flat-Sta}, \ts
\ref{lem:CRLE-Flat} \ts and \ts \ref{l:Quad-CRLE},
we conclude that \. {\sc QuadRLE} \. reduces to \. {\sc EqualityStanley}$_k$\ts.  We have:
\begin{equation}\label{eq:together-verify}
	\NP^{\<\text{\sc VerRLE}\>}
	\. \subseteq \. \NP^{\<\text{\sc QuadRLE} \>} \. \subseteq \. \NP^{\<\text{\sc EqualityStanley}_k\>},
\end{equation}
where the first inclusion is the Verification Lemma~\ref{l:Verify-Quart}.
Now, suppose \ts $\text{\sc EqualityStanley}_k \in \PH$.  Then \ts {\sc EqualityStanley}$_k \in \ts \Sigmap_m$ \ts
for some~$m$.  Combining \eqref{eq:together-LE} and \eqref{eq:together-verify}, this implies:
\begin{equation}\label{eq:together-collapse}
	\PH \. \subseteq \. \NP^{\<\text{\sc EqualityStanley}_k\>} \. \subseteq \.  \NP^{\Sigmap_m} \. \subseteq \. \Sigmap_{m+1}\,,
\end{equation}
as desired.
\end{proof}

As a byproduct of the proof, we get the same conclusion for the
intermediate problems.  This result is
potentially of independent interest (cf.~\cite{CP23}).
\smallskip

\begin{cor}\label{c:summary-notin-PH}
Problems \. {\sc VerRLE},  \. {\sc QuadRLE}\., \. {\sc CRLE} \. and \. {\sc FlatLE$_0$} \.
are not in \ts $\PH$, unless \ts $\PH=\Sigmap_m$ \ts for some~$m$.
\end{cor}

\smallskip

\begin{proof}[Proof of Theorem~\ref{t:main-AF}]
The result follows from Proposition~\ref{p:AF-Sta} and Theorem~\ref{t:main-Sta}.
\end{proof}


\begin{proof}[Proof of Corollary~\ref{c:main-AF-stab}]
By the ``Bonnesen type'' assumption, we have
$$\big\{\xi(\cdot) =^?0\big\} \ \ \Longleftrightarrow  \ \
\big\{\de(\cdot) =^?0\big\} \, = \, \text{\sc EqualityAF}\ts.
$$
Since computing \ts $\xi$ \ts is in~$\FP$, we have \ts {\sc EqualityAF} $\in \P$.
Then \eqref{eq:together-collapse} for $k=2$, and Proposition~\ref{p:AF-Sta} give:
\begin{equation}\label{eq:together-stab}
	\PH \. \subseteq \. \NP^{\<\text{\sc EqualityStanley}_2\>} \. \subseteq \.
\NP^{\<\text{\sc EqualityAF}\>} \. \subseteq \.  \NP^{\P} \. = \. \NP,
\end{equation}
as desired.
\end{proof}


\begin{proof}[Proof of Corollary~\ref{c:main-Stanley-not-SP}]
Suppose \ts $\phi_k \in \SP$.  By definition, we have: \ts
$$
\big\{\. \Phi_{\bz\bc}(P, x,a) \, \ne^? \, 0\.\big\} \. \in \. \NP.
$$
In other words, we have
 \ts {\sc EqualityStanley}$_k \in \coNP$.
%
Then \eqref{eq:together-collapse} gives:
$$
	\PH \. \subseteq \. \NP^{\<\text{\sc EqualityStanley}_k\>} \. \subseteq \.  \NP^{\ts \coNP} \. = \. \Sigmap_{2}\,,
$$
as desired.
\end{proof}

\medskip

\section{AF equality from Stanley equality} \label{s:AF}

\subsection{Slices of order polytopes} \label{ss:AF-slices}
Let \ts $P=(X,\prec)$ \ts be a poset on \ts $|X|=n$ \ts elements.
Recall the construction of order polytopes \ts $\cO_P\subseteq [0,1]^n$ \ts
given in~\eqref{eq:order-def}.  Fix \. $z_1\prec \ldots \prec z_k$
\. and \. $1 \le c_1 < \ldots < c_k\le n$\ts.
Denote \ts $Z:=\{z_1,\ldots,z_k\}$ \ts and let \ts $Y:=X\sm Z$.
For all \. $0\le i \le k$, consider \ts the following \defn{slices} \ts of the order
polytopes:
$$
\iS_i \, := \, \cO_P \. \cap \. \{\al_x=0 \, : \, x \preccurlyeq z_i, \. x \in X \}
\. \cap \. \{\al_x=1 \, : \, x \succcurlyeq z_{i+1}, \. x \in X \}. 
$$
Here the conditions \ts $x \preccurlyeq z_i$ \ts and \ts $x \succcurlyeq z_{i+1}$ \ts are
vacuous when \ts $i=0$ \ts and \ts $i=k$, respectively.  Note that \ts $\dim \iS_i \le n-k$ \ts
for all \ts $0\le i \le k$, since \ts $\al_x$ \ts is a constant on \ts $\iS_i$ \ts
for all \ts $x\in Z$.\footnote{In geometric language, slices \ts $\iS_i$ \ts are sections
of the order polytope \ts $\cO_P$ \ts with a $k$-dimensional affine subspace.}  The same
argument implies that these slices are themselves order polytopes of subposets of~$P$,
a fact we do not need.  Instead, we need the following simple result:

\smallskip

\begin{lemma}\label{l:slices-TU}
Slices \. $\iS_i$ \ts are \ts $\TU$-polytopes.
\end{lemma}

\begin{proof}
Write \ts $\iS_i$ \ts in the form \. $A \cdot (\al_y)_{y\in Y} \le \bbi$.
Observe that \ts $A$ \ts has \ts $\{-1,0,1\}$ \ts entries, and so does~$\bbi$.
Every square submatrix $B$ of~$A$ corresponds to taking a subposet with added
rows of $0$'s, or with rows of $0$'s and a single~$\pm 1$.  By definition of \ts $\cO_P$, we can
rearrange columns in~$B$ to make it upper triangular.  Thus, $\det(B) \in \{-1,0,1\}$,
as desired.
\end{proof}

\smallskip

\subsection{Proof of Proposition~\ref{p:AF-Sta}} \label{ss:AF-proofs}
Denote by \ts $\Ec_{\zb \bc}(P)$ \ts the set of all linear extensions \ts $f\in \Ec(P)$,
such that \ts $f(z_i)=c_i$ \ts for all~$i$, and let \ts $\aN_{\bz\bc}(P):=|\Ec_{\zb \bc}(P)|$.

Let \ts $\iS_0,\ldots,\iS_k\ssu \rr^n$ \ts be the slices defined above, and note that the dimension \ts
$\dim\<\iS_0,\ldots,\iS_k\> $ \ts  of the subspace spanned by vectors in \ts $\iS_0,\ldots, \iS_k$ \ts is equal to  $n-k$.
Stanley's original proof of~\eqref{eq:Sta} is based on the following key observation:

\smallskip

\begin{lemma}[{\rm \cite[Thm~3.2]{Sta-AF}}{}]  \label{l:Sta-AF}
Let \. $z_1\prec \ldots \prec z_k$
\. and \. $1 \le c_1 < \ldots < c_k\le n$. We have:
\begin{equation}\label{eq:Sta-pol}
\iV\big(\underbrace{\iS_0 \. ,\. \ldots \. , \. \iS_0}_{\text{$c_1-1$  times}}, \.
\underbrace{\iS_1 \. ,\. \ldots \. , \. \iS_1}_{\text{$c_2-c_1-1$  times}}, \. \ldots \. ,
\underbrace{\iS_k \. ,\. \ldots \. , \. \iS_k}_{\text{$n-c_k$  times}}\big) \  =  \ \frac{1}{(n-k)!} \, \aNr_{\bz\bc}(P).
\end{equation}
\end{lemma}

\smallskip

Now let \ts $z_i \gets x$ \ts and \ts $c_i\gets a$ \ts for some~$i$, such that \.
$1 \le c_1 < \ldots < c_k\le n$.  By Lemma~\ref{l:Sta-AF}, the AF inequality \eqref{eq:AF}
becomes \eqref{eq:Sta}.  By Lemma~\ref{l:slices-TU}, slices \ts $\iS_i\ssu \rr^n$ \ts are
TU-polytopes defined by \ts $O(n^2)$ \ts inequalities.
This gives the desired reduction. \qed

\medskip



\section{Stanley equality from flatness} \label{s:Flat-Sta}

\subsection{Ma--Shenfeld poset notation} \label{ss:main-MS}
Recall the following terminology from \cite{MS22}.
Let  \. $s\in \{-1,0,1\}$\..
For any  \. $f \in \Ec_{\bz\bc}(P,x,a+s)$,
the \defnb{companions} \ts in~$\ts f$ \ts
are the elements in
\[ \Com(f) \, := \,
\big\{f^{-1}(a-1), \. f^{-1}(a), \. f^{-1}(a+1)  \big\} \. - \. x. \]
Note that \. $|\Com(f)| =2$ \. for all $s$ as above.
Let the \defn{lower companion} \. $\lc(f)\in \Com(f)$ \. be
the companion with smaller of the two values in~$f$.  Similarly,
let the \defnb{upper companion} \. \. $\uc(f)\in \Com(f)$ \. be
the companion with larger of the two values in~$f$.   Denote by \ts
$\cC(x)\ssu X$ \ts the set of elements \ts $y\in X$ comparable to~$x$, i.e.\
\.  $\cC(x) := \{y \in X \, : \, x \prec y \ \text{or} \ x \succ y\}$.

\smallskip

\subsection{Proof of Lemma~\ref{l:Flat-Sta}}\label{s:Flat-Sta-proof}
%
%
%
Let \. $P=(X,\prec)$, and let \ts $x$, \ts $a$, \ts $\zb=(z_1,\ldots,z_k)$
\ts and \ts $\cb=(c_1,\ldots,c_k)$ \. be an instance of \ts {\sc FlatLE$_k$}
\ts as in Lemma~\ref{l:Flat-Sta}.  To prove the reduction in the lemma, we construct a poset
\ts $Q=(Y,\prec)$ \ts for which \ts $P$ \ts
is a subposet, and \ts $x$, \ts $b$, \ts $\yb$ \ts and \ts $\bx$,
which give the desired instance \. {\sc EqualityStanley$_{k+2}$}\ts.

Without loss of generality, we can assume that \ts $\min(P)=\{z_0\}$ \ts
and  \ts $\max(P)=\{z_{k+1}\}$. In other words, assume that there are
elements \. $z_0,z_{k+1} \in X$ \. such that \. $z_0 \preccurlyeq y \preccurlyeq z_{k+1}$ \.
for all \ts $y \in X$.

Let  \. $\aM_1,\aM_2, \aM_3$ \. be given by
\begin{align*}
	\aM_1 \ &:= \  \big|\big\{ f \in \Ec_{\zb\bc}(P,x,a) \, : \,  f^{-1}(a+1) \succ x \big\} \big|\.,	\\
	\aM_2 \ &:= \  \big|\big\{ f \in \Ec_{\zb\bc}(P,x,a+1)  \, : \,  f^{-1}(a) \prec x \big\} \big|\.,\\
	\aM_3 \ &:= \  \big|\big\{ f \in  \Ec_{\zb\bc}(P,x,a)  \, : \,  f^{-1}(a+1) \. \| \. x \big\} \big| \\
	& \qquad = \  \big|\big\{ f \in  \Ec_{\zb\bc}(P,x,a+1)  \, : \,  f^{-1}(a) \. \| \. x \big\} \big|\..
\end{align*}
Note that the two sets in the definition of \ts $\aM_3$ \ts are in bijection with  each other via the map that swaps \. $f(a)$ \. with \. $f(a+1)$\..
It then follows from here that
\[
\aN_{\zb \bc}(P,x,a) \, = \,  \aM_1 \. + \. \aM_3 \qquad \text{and} \qquad
\aN_{\zb \bc}(P,x,a+1) \, = \,  \aM_2 \. + \. \aM_3\ts.
\]
This implies that
\begin{equation}\label{eq:cmkt-1}
	\aN_{\zb \bc}(P,x,a) \, = \, \aN_{\zb \bc}(P,x,a+1) \qquad \Longleftrightarrow \qquad \aM_1 \ = \ \aM_2\ts.
\end{equation}

Now, let \. $Q=(Y,\prec)$ \. be the poset \ts $P+C_3$\ts, i.e.
\. $Y:=X \cup \{u,v,w\}$ \. and with the additional relation \. $u \prec v \prec w$ \.
and \. $\{u,v,w\}$ \. is incomparable to all elements in~$X$.
Let \. $\ell :=\max \{i \. : \. c_i< a\}$ \. be the maximal index
such that the corresponding element in $\bz$ is less than~$a$.
Let \. $b := a+2$, and let
\begin{align*}
	&\yb \ := \ (z_1,\ldots, z_{\ell}, u, w, z_{\ell+1}, \ldots, z_{k}) \. \in \. Y^{k+2}, \\
	 & \bbb \ := \ (c_1,\ldots, c_\ell, a, a+4, c_{\ell+1}+3, \ldots, c_{k}+3) \. \in \. \nn^{k+2}.
\end{align*}

In the notation above, for \. $s\in \{-1,0,1\}$ \. and \. $f \in \Ec_{\yb\bbb}(Q,x,b+s)$,
the companions \ts in~$\ts f$ \ts
are the elements in
\[ \Com(f) \, := \,
\big\{f^{-1}(b-1), \. f^{-1}(b), \. f^{-1}(b+1)  \big\} \. - \. x\ts .
\]
Let\footnote{We warn the reader that from this point on our notation is substantially
different from that in~\cite{MS22}.}
\begin{align*}
	\cF(b,\comp,\inc) \ &:= \ \big\{f \in  \Ec_{\yb\bbb}(Q,x,b) \, : \, \lc(f)\in \cC(x), \. \uc(f)\not\in \cC(x) \big\},\\
		\cF(b,\inc,\comp) \ &:= \ \big\{f \in  \Ec_{\yb\bbb}(Q,x,b) \, : \, \lc(f)\notin \cC(x), \. \uc(f)\in \cC(x) \big\},\\
			\cF(b,\comp,\comp) \ &:= \ \big\{f \in  \Ec_{\yb\bbb}(Q,x,b) \, : \, \lc(f)\in \cC(x), \. \uc(f)\in \cC(x) \big\},\\
		\cF(b,\inc,\inc) \ &:= \ \big\{f \in  \Ec_{\yb\bbb}(Q,x,b) \, : \, \lc(f)\not\in \cC(x), \. \uc(f)\not\in \cC(x) \big\},	
\end{align*}
and we write \. $\aF(b,\cdot, \cdot) \. := \. |\cF(b,\cdot,\cdot)|$\ts.
Note that by construction it follows that, for all \ts $f \in \cF(b,\cdot,\cdot)$, we have
\[ b-2 \ = \ f(u)  \ < \  f(v) \ < \ f(w) \ = \ b+2, \]
so \ts $f(v) \in \{b-1,b,b+1\}$\ts, and thus $v$ will always be a companion in~$f$.
Sets \ts $\cF(b+1,\ast,\ast)$ \ts and \ts $\cF(b-1,\ast,\ast)$ \ts are defined
analogously.

\smallskip

\begin{claim}  We have: {\rm
\begin{alignat*}{2}\label{eq:yskn-1}
	& \aF(b,\comp,\inc) \ = \ \aM_2\,, \qquad && \aF(b,\inc,\comp) \ = \ \aM_1\ts,\\
	& \aF(b,\comp,\comp) \ = \ 0\,, \qquad && \aF(b,\inc,\inc) \ = \ 2\ts\aM_3\ts,\\
	& \aF(b+1,\comp,\inc) \ = \ \aM_2\,, \qquad && \aF(b+1,\inc,\comp) \ = \ \aM_2\ts,\\
& \aF(b+1,\comp,\comp) \ = \ 0\,, \qquad && \aF(b+1,\inc,\inc) \ = \ 2\ts\aM_3\ts,\\
	& \aF(b-1,\comp,\inc) \ = \ \aM_1\,, \qquad && \aF(b-1,\inc,\comp) \ = \ \aM_1\ts,\\
& \aF(b-1,\comp,\comp) \ = \ 0\,, \qquad && \aF(b-1,\inc,\inc) \ = \ 2\ts\aM_3\ts.
\end{alignat*}}
\end{claim}

\begin{proof}
We  only compute the values \ts $\aF(b,\ast,\ast)$, as proof of the other cases is analogous.
Denote by \ts $\Ec_{\zb \bc}(P)$ \ts the set of all linear extensions \ts $f\in \Ec(P)$,
such that \ts $f(z_i)=c_i$ \ts for all~$i$.

Let \. $\psi :\Ec_{\yb \bbb}(Q) \to \Ec_{\zb \bc}(P)$ \. be the map given by \ts $\psi(f) = g$, where
\begin{equation*}
	g(s) \ := \
\begin{cases}
	f(s) & \text{ if } \ \  f(s)<f(u),\\
	f(s)-1 & \text{ if } \ \ f(u) < f(s) < f(v),\\
	f(s)-2 & \text{ if } \ \ f(v) < f(s) < f(w),\\
	f(s)-3 & \text{ if } \ \ f(s)>f(w),
\end{cases}
\end{equation*}
for all \ts $s\in X$.
It follows from the definition of \ts $\lc(f)$ \ts and \ts $\uc(f)$,  that
\[ 	
\cF(b,\comp,\inc) \ = \ \big\{f \in  \Ec_{\yb\bbb}(Q,x,b) \, : \, f^{-1}(b-1) \prec x, \. f^{-1}(b+1)=v \big\}
\., \]
It then follows that \ts $\vp$ \ts restricted to \ts $\cF(b,\comp,\inc)$ \ts is a bijection onto
\[ \big\{g \in  \Ec_{\bz\bc}(P,x,a+1) \, : \, g^{-1}(a+1) \prec x \big\},
\]
which gives us
$\aF(b,\comp,\inc)=\aM_2$.
Similar arguments gives \. $\aF(b,\inc,\comp)=\aM_1$\ts.
Note that \. $\aF(b,\comp,\comp)=0$, because
$v$ is always a companion in $f$ but
 $v \. \| \. x$ by definition.
Note also that
\begin{equation*}
\aligned
\cF(b,\inc,\inc) \ & =  \ \big\{f \in  \Ec_{\yb\bbb}(Q,x,b) \ : \ f^{-1}(b-1)  \. \|  \. x, \. f^{-1}(b+1)=v  \big\} \\
& \hskip1.cm \cup \
\big\{f \in  \Ec_{\yb\bbb}(Q,x,b) \ : \ f^{-1}(b+1) \. \|  \. x, \. f^{-1}(b-1)=v \big\}.
\endaligned
\end{equation*}
It then follows that \ts $\psi$ \ts restricted to \ts $\cF(b,\comp,\inc)$ \ts is a bijection onto
\[
\big\{g \in  \Ec_{\bz\bc}(P,x,a+1) \, : \, g^{-1}(a) \. \| \. x \big\} \ \cup \
\big\{g \in  \Ec_{\bz\bc}(P,x,a) \, : \, g^{-1}(a+1) \. \| \. x \big\},
\]
which gives \. $\aF(b,\inc,\inc)=2\ts\aM_3$.
This finishes proof of the claim.
\end{proof}

\smallskip

By the claim, we have:
\begin{equation*}
\begin{split}
	\aN_{\yb \bbb}(Q,x,b) \ &= \ \aF(b,\comp,\inc) \ + \ \aF(b,\inc,\comp) \ + \ \aF(b,\comp,\comp) \ + \ \aF(b,\inc,\inc)\\
	\ &= \ \aM_2 \  + \  \aM_1 \ + \  2\aM_3.
\end{split}
\end{equation*}
Similarly, we have:
\begin{equation*}
	\begin{split}
		\aN_{\yb \bbb}(Q,x,b+1) \ &= \  2 \ts\aM_2  \ + \  2\ts\aM_3\ts, \\
				\aN_{\yb\bbb}(Q,x,b-1) \ &= \  2 \ts\aM_1  \ + \  2\ts\aM_3\ts.
	\end{split}
\end{equation*}
We conclude:
\begin{align*}
& \aN_{\yb\bbb}(Q,x,b)^2  \. - \. \aN_{\yb \bbb}(Q,x,b+1) \cdot \aN_{\yb\bbb}(Q,x,b-1)   \\
& \hskip1.cm = \. (\aM_1+\aM_2+2\aM_3)^2 \. - \. 4(\aM_1+\aM_3)(\aM_2+\aM_3) \, = \,  (\aM_1-\aM_2)^2.
\end{align*}
This implies that
\begin{equation}\label{eq:cmkt-2}
\aN_{\yb\bbb}(Q,x,b)^2  \, = \, \aN_{\yb\bbb}(Q,x,b+1) \cdot \aN_{\yb\bbb}(Q,x,b-1)
\quad \Longleftrightarrow \quad \aM_1=\aM_2\ts.
\end{equation}
Lemma~\ref{l:Flat-Sta} now follows by combining \eqref{eq:cmkt-1} and \eqref{eq:cmkt-2}. \qed



\medskip

\section{Flatness from the quadruple relative ratio}\label{s:Quad-Fla}
%
Recall several key definitions from Section~\ref{s:roadmap}.
Let \. $\aN(R,z,c)$ \. be the number of linear extensions \ts
$f\in \Ec(R)$ \ts for which \. $f(z)=c$.  Similarly, let
\begin{equation}\label{eq:Flat-RLE}
	\text{\sc FlatLE}_0 \ := \ \big\{\. \aN(R,z,c) \. =^? \. \aN(R,z,c+1)\.\big\},
\end{equation}

\nin
where \ts $R=(Z,\prec^\circ)$ is a finite poset on \ts $|Z|=\ell$ \ts elements, \ts $z\in Z$ \ts
and \. $1\le c \le \ell$.
Finally, let
 \begin{equation}\label{eq:CRLE} 
	\text{\sc CRLE} \ := \ \big\{ \.	
	\rho(P,x) \, =^? \, \rho(Q,y)
	\. \big\},
\end{equation}
where \ts $P=(X,\prec)$, \ts $Q=(Y,\prec')$ \ts are posets, and \ts $x\in \min(P)$,
\ts $y \in \min(Q)$.

\smallskip

\subsection{One poset from two} \label{ss:Flat-ratio-to-one}
The following result give a quantitative version\footnote{We do not actually
need the precise bounds below, other than the fact that they are at most polynomial.
However, these bounds help	to clarify the construction.} of Lemma~\ref{lem:CRLE-Flat}.

\smallskip

\begin{thm}
	\label{thm:CRLE-Flat} \. {\sc CRLE} \.
	reduces to \. {\sc FlatLE$_0$}\ts. More precisely, suppose we have
	a poset \ts $P=(X,\prec)$ \ts on \ts $n=|X|$ \ts elements,
	a poset \ts $Q=(Y,\prec')$ \ts on \ts $m=|Y|$ \ts elements,
	and \ts $x\in \min(P)$, \ts $y\in \min(Y)$.
	Then there exists a  polynomial time construction of a poset \.
	$R=(Z,\prec^\circ)$ \. on \. $\ell:=|Z| =m+n$ \. elements,  \ts $z\in Z$\ts,
and \ts $c\in[\ell]$\ts, such that \ \eqref{eq:Flat-RLE} \, $\Leftrightarrow$ \, \eqref{eq:CRLE}.
\end{thm}


\begin{proof}
	Let \ts $P^\ast=(X,\prec^\ast)$ \ts be the dual poset of~$P$.
	Define \ts $R=(Z,\prec^\circ)$ \ts to be a poset on
	\[  Z \ := \ (X-x) \cup (Y-y) \cup \{w,z\}\ts,\]
	where \ts $w,z$ \ts are two new elements.
	Let the partial order \. $\prec^\circ$ \. coincide with \ts $\prec^\ast$ \ts on \ts $(X-x)$,
	and with \ts $\prec'$ \ts on \ts $(Y-y)$, with  additional relations
	\begin{align}
		\label{eq:koko-1} & p \. \prec^\circ \.  z \. \prec^\circ \. q\., \quad \text{ for all } \ \ p \in X-x, \ q \in Y-y,\\
		\label{eq:koko-2} & p \. \prec^\circ w \ \ \text{ if and only if } \ \  x \prec p \., \quad \text{ for all } \ \ p \in X-x, \\
		\label{eq:koko-3} & w \. \prec^\circ q \ \  \text{ if and only if } \ \ y \prec' q\., \quad \text{ for all } \ \ q \in Y-y.
	\end{align}
That is, we are taking the series sum   \ts $(P^*-x) \. \oplus \. \{z\} \. \oplus \. (Q-y)$,
then adding an element \ts $w$ \ts to emulate \ts $x$ \ts in \ts $P$ \ts for $f(w) <f(z)$,
where $f \in \Ec(R)$, and emulate $y$ in $Q$ when $f(w) > f(z)$.
	It then follows from a direct calculation that
	\[ \aN(R,z,n+1) \, = \, e(P) \. \cdot \. e(Q-y). \]
	Indeed, by \eqref{eq:koko-1}, for every \. $f \in \Nc(R,z,n+1)$ \. we have:
	\[
	\big\{f^{-1}(1),\ldots, f^{-1}(n) \big\} \, = \, X-x +w, \qquad  \big\{f^{-1}(n+2),\ldots, f^{-1}(m+n) \big\} \, = \, Y-y\ts.
	\]
	These two labelings define a linear extension of \ts $P^*$ \ts after a substitution \. $w\gets x$ \. given by \eqref{eq:koko-2},
	and a linear extension  of \ts $Q-y$, and it is clear that this construction defines a bijection.  By an analogous argument, we have:
	\[ \aN(R,z,n) \, = \, e(P^*-x) \. \cdot \. e(Q). \]
Set \ts $c\gets n$. Combining these two observations, we get
	\[
	\frac{\aN(R,z,c+1)}{\aN(R,z,c)} \ = \ \frac{\aN(R,z,n+1)}{\aN(R,z,n)} \ = \ \frac{\agr(P,x)}{\agr(Q,y)}\,,
	\]
	which gives the desired reduction and proves the result.
\end{proof}

\medskip

\subsection{Two posets from four} \label{ss:Quad-Flat-ratio}
Now recall the decision problem
\begin{equation}\label{eq:Quart-RLE}
\text{\sc QuadRLE} \ := \ \big\{ \.	
	\ag(P_1,x_1) \cdot \ag(P_2,x_2) \, =^? \,
	\ag(P_3,x_3) \cdot \ag(P_4,x_4)
	\. \big\}.
\end{equation}
The following result give a quantitative version of Lemma~\ref{l:Quad-CRLE}.

\smallskip

\begin{thm}\label{thm:gCF}
	$\text{\sc QuadRLE}$ \ts reduces to \ts {\sc CRLE}.  More precisely,
	for every \. $P_i = (X_i,\prec_i)$ \.  posets on \. $n_i= |X_i|$ \. elements,
	and \ts $x_i\in \min(P_i)$, \ts $1\le i \le 4$,
	there exists a  polynomial time construction of a poset \. $P=(X,\prec)$ \.
	on \. $n:=|X| \le n_1+\max\{n_2,n_3\}+1$ \. elements, of a poset \.  $Q=(Y,\prec')$ \. on \.
	$m:=|Y|\le n_4+ \max\{n_2,n_3\}+1$ \. elements,
	such that \ \eqref{eq:CRLE} \, $\Leftrightarrow$ \,  \eqref{eq:Quart-RLE}.
\end{thm}
\smallskip

We now build toward the proof of this theorem. 

\smallskip

\begin{lemma}\label{lem:Sta-ver-pre}
	Let \ts $P=(X,\prec)$ \ts and \ts $Q=(Y,\prec')$ \ts be posets with \ts $m=|X|$ \ts
	and \ts $n=|Y|$ \ts elements, respectively.  Let \ts $x \in \min(P)$ \ts and \ts $y\in \min(Q)$.
	Then there exists a poset \ts $R=(Z,\prec^\circ)$ \ts and \ts $z \in \min(P)$, such that \. $|Z| = m+n+1$ \. and
	\[ \rho(R,z) \ = \  m \ + \  \left(1  + \tfrac{\rho(Q,y)}{\rho(P,x)}\right)^{-1}\..  \]
\end{lemma}

\smallskip

\begin{proof}
	Let \ts $P^\ast= (X,\prec^\ast)$ \ts denotes the dual poset to~$P$.
	Let  \. $R:=(Z,\prec^\circ)$ \. be given by
	\[ Z \ := \  (X-x) \. \cup  \. (Y-y)  \. \cup \. \{v,w,z\}.  \]
	where \ts $\prec^\circ$ \ts inherits the partial order \ts $\prec^*$ \ts on \ts $X-x$\ts,
	the partial order \ts $\prec'$ \ts on \ts $Y-y$\ts,
	and with additional relations:
	\begin{align*}
		& p \, \prec^\circ \, v \. \prec^\circ \. q \qquad \forall \, p \in X-x, \, \ q \in Y-y, \\
		& p \, \prec^\circ \, w   \ \ \, \Leftrightarrow  \ \ \,   p \prec^* \. x \quad \forall \, p \in X-x,\\
		& q \, \succ^\circ \, w    \ \ \,  \Leftrightarrow   \ \ \,   q \succ' \. y \quad \forall \, y \in Y-y,\\
		& z \,  \|_{\prec^\circ} \ p  \ \ \,  \forall \ p \in X-x,  \quad   z \,  \prec^\circ \ q \quad \forall \, q \in Y-y, \\
		& z \, \prec^\circ \, v,  \quad  z \, \|_{\prec^\circ} \ w.
	\end{align*}
That is, we are taking the series sum  \ts $(P^*-x) \. \oplus \. \{v\} \. \oplus \. (Q-y)$\ts,
then adding an element \ts $w$ \ts to emulate \ts $x$ \ts in $P$ \ts for all \ts $f(w) <f(v)$,
emulate \ts $y$ \ts in \ts $Q$ \ts for all \ts $f(w) > f(v)$,
and finally adding \ts $z$ \ts to track the value of \ts $f(v)$.
Here the linear extension \ts $f \in \Ec(R)$ \ts in each case.
By construction, we have either \. $f(v)=m+1$ \. or \. $f(v)=m+2$.
	
	\smallskip
	
	\nin
	{\bf Claim.} \. We have:	
	\begin{align}\label{eq:Iching-1}
		e(R) \ = \   m \. e(P-x)  \. e(Q) \. + \.   (m+1) \. e(P)  \. e(Q-y).
	\end{align}
	
	\begin{proof}[Proof of claim]
		Let us show that the first term \. $m \. e(P-x)  \. e(Q) $ \.
		is the number of  linear extensions \ts $f \in \Ec(R)$ \ts s.t.\ \ts $f(v)=m+1$.
		For such $f$ we have:
		\[ \aligned
		\big\{f^{-1}(1),\ldots, f^{-1}(m) \big\} \ & = \ X- x + z, \\
		\big\{f^{-1}(m+2),\ldots, f^{-1}(m+n+1) \big\} \ & = \ Y  -y + w.
		\endaligned
		\]
		Note that the restriction of  \ts $f$ \ts to \. $\big\{f^{-1}(1),\ldots, f^{-1}(m) \big\}$ \.
		defines a  linear extension of \ts $(P^*-x +z)$.
		Additionally, note that the restriction of  \ts $f$ \ts to \. $\big\{f^{-1}(m+2),\ldots, f^{-1}(m+n+1) \big\}$ \. defines
		a  linear extension of \ts $Q$.
		It is also clear that  this construction defines a bijection.
		 In total, we have \. $e(P^*-x +z)\. e(Q) = m \. e(P-x)\. e(Q)$ \. linear extensions~$\ts f$ \ts
		as above.
		
		Similarly, let us show that the second term \. $(m+1) \. e(P)  \. e(Q-y)$ \.
		is the number of linear extensions \ts $f \in \Ec(R)$ \ts s.t.\ \ts $f(v)=m+2$.
		For such $f$ we have:
		\[ \aligned
		\big\{f^{-1}(1),\ldots, f^{-1}(m+1) \} \ & = \ X-x + w +z, \\
		\big\{f^{-1}(m+3),\ldots, f^{-1}(m+n+1) \big\} \ & = \ Y  -y.
		\endaligned\]
		Note that the restriction of \ts $f$ \ts to \. $\big\{f^{-1}(1),\ldots, f^{-1}(m+1) \big\}$ \.
		defines a  linear extension of \ts $(P^* + z)$.  Additionally, note that
		the restriction of \ts $f$ \ts to \. $\big\{f^{-1}(m+3),\ldots, f^{-1}(m+n+1) \big\}$ \. defines a
		linear extension of \ts $(Q-y)$.
				It is also clear that  this construction defines a bijection.
		In total, we have \. $e(P^* + z)\. e(Q-y) = (m+1) \. e(P) \. e(Q-y)$
		linear extensions~$\ts f$ \ts as above.  This completes the proof. 
	\end{proof}
	
	\smallskip
	
	Following the argument in the claim we similarly have:
	\begin{align}\label{eq:Iching-2}
		e(R-z) \ = \   \. e(P-x)  \. e(Q) \. + \.   \. e(P)  \. e(Q-y).
	\end{align}
	Indeed, the term \. $e(P-x)  \. e(Q)$ \. is the number of linear extensions \ts $f \in \Ec(R)$
	for which \. $f(v)=m$, and the term \. $e(P)  \. e(Q-y)$ \. is the number of linear extensions \ts $f \in \Ec(R)$
	for which \. $f(v)=m+1$.  We omit the details.
	
	\smallskip
	
	Combing \eqref{eq:Iching-1} and \eqref{eq:Iching-2},
	we have:
	\begin{align*}
		\rho(R,z) \ = \  m \. + \. \frac{e(P) \. e(Q-y)}{e(P-x)  \. e(Q) \. + \.   \. e(P)  \. e(Q-y)} \ = \  m \ + \  \left(1  + \frac{\rho(Q,y)}{\rho(P,x)}\right)^{-1},
	\end{align*}
	as desired.
\end{proof}

\smallskip

\begin{lemma}\label{lem:KS-1}
	Let \ts $P=(X,\prec)$ \ts be a  poset on \ts $n=|X|$ elements, and let \ts $x\in \min(P)$.
	Then there exists a  poset \ts $Q=(Y,\prec')$ \ts  and an element \ts $y \in \min(Q)$,
	such that \ts $|Y|=n+1$ \ts and
	\[  \agr(Q,y) \ = \  1 \, + \, \frac{1}{\agr(P,x)}\,.\]
\end{lemma}

\begin{proof}
	Let \ts $Y:=X+z$, and let \ts $\prec'$ \ts coincide with \ts $\prec$ \ts on~$P$,
	with added relations
	\begin{align*}
		z \ts \prec' \ts u \quad \text{for all} \ \ u \in X-x \ts, \quad  \text{and} \quad z \ts \| \ts x\ts.
	\end{align*}
	Note that \ts $z\in \min(Q)$\ts.
	Note also that
	\[ e(Q-z) \, = \, e(P) \quad \text{and} \quad e(Q) \, = \,  e(P) \. + \. e(P-x), \]
	since for every \ts $f\in \Ec(Q)$ \ts we either have \ts $f(z)=1$,
	or \ts $f(z)=2$ and thus \ts $f(x)=1$.
	We now take \. $y \gets z$, and observe that
	\[  \ag(Q,y) \ = \  \frac{e(Q)}{e(Q-z)} \ = \ \frac{e(P) + e(P-x)}{e(P)} \ = \ 1 \, + \, \frac{1}{\agr(P,x)} \,,   \]
	as desired.
\end{proof}

\smallskip

\smallskip
\begin{lemma}\label{lem:KS-2}
	Let \ts $P=(X,\prec)$ \ts be a  poset on \ts $n=|X|$ elements, and let \ts $x\in \min(P)$.
	Then there exists a  poset \ts $Q=(Y,\prec')$\ts, and \ts $y \in \min(Q)$,
	such that \ts $|Y|=n+1$ \ts and
	\[  \agr(Q,y) \ = \  1 + \agr(P,x).\]
\end{lemma}

\smallskip

\begin{proof}
	Let $Q$ be as in the proof of Lemma~\ref{lem:KS-1}.
	Note that \ts $x\in \min(Q)$, and that
	\[ e(Q-x) \, = \, e(P-x), \]
	since \ts $z$ \ts is the unique minimal element in \ts $Q-x$.
	We now take \. $y \gets x$,    and observe that
	\[  \ag(Q,y) \ = \  \frac{e(Q)}{e(Q-x)} \ = \ \frac{e(P)+e(P-x)}{e(P-x)} \ = \ 1 \. + \. \ag(P,x),   \]
	as desired.
\end{proof}

\smallskip

\begin{proof}[Proof of Theorem~\ref{thm:gCF}]
	By symmetry, we will without loss of generality assume that $n_2 \geq n_3$.
	By applying Lemma~\ref{lem:Sta-ver-pre} followed by applying Lemma~\ref{lem:KS-2}  for \ts  $n_2-n_3$ \ts  many times,
	we get a poset \ts $P=(X,\prec)$ \ts and $x \in \min(P)$
	such that
	\[ \rho(P,x) \ = \   (n_2- n_3) \ + \  \Big( n_3 \ + \  \left(1  + \tfrac{\rho(P_1,x_1)}{\rho(P_3,x_3)}\right)^{-1} \Big) \ = \ n_2 \. + \. \left(1  + \tfrac{\rho(P_1,x_1)}{\rho(P_3,x_3)}\right)^{-1}. \]
	Additionally, poset~$P$ has
	\[ |X| \ = \ n \ = \ (n_1 + n_3+1) + n_2-n_3 \  = \ n_1+ \max\{n_2,n_3\}+1 \quad \text{elements}.  \]
	On the other hand, by Lemma~\ref{lem:Sta-ver-pre} we get a poset $Q$ and $y \in \min(Q)$, s.t.\
	such that
	\[ \rho(Q,y) \ = \    n_2 \ + \  \left(1  + \tfrac{\rho(P_4,x_4)}{\rho(P_2,x_2)}\right)^{-1}, \]
	and with
	\[ m \ = \ n_2 + n_4+1 \  = \ n_4+ \max\{n_2,n_3\}+1.   \]
	
	It now follows that
	\[ \rho(P,x) \. = \.  \rho(Q,y) \qquad \Longleftrightarrow \qquad \frac{\rho(P_1,x_1)}{\rho(P_3,x_3)} \. = \. \frac{\rho(P_4,x_4)}{\rho(P_2,x_2)}\,, \]
	as desired.
\end{proof}

\medskip

%
%
%
%



\section{Verification lemma} \label{s:verify}

The proof of the Verification Lemma~\ref{l:Verify-Quart} is different
from other reductions which are given by parsimonious bijections.
Before proceeding to the proof, we need several technical and
seemingly unrelated results.

\smallskip

\subsection{Continued fractions}\label{ss:verify-CF}
Given \. $a_0\geq 0$\., \. $a_1, \ldots, a_s \in \nz $\., where  \ts $s \geq 0$,
the corresponding \defnb{continued fraction} \ts is defined as follows:
\[ [a_0\ts ; \ts a_1,\ldots, a_s] \ := \ a_0  \. +  \.  \cfrac{1}{a_1  \. +  \. \cfrac{1}{\ddots  \ +  \.  \frac{1}{a_s}}}
\]
Numbers \ts $a_i$ \ts are called \ts \defn{quotients},
see e.g.\ \cite[$\S$10.1]{HardyWright}.  We refer to \cite[$\S$4.5.3]{Knuth98}
for a detailed asymptotic analysis of the quotients in connection with
the Euclidean algorithm, and further references.  The following technical
result is key in the proof of the Verification Lemma.

\smallskip

\begin{prop}[{\rm cf.\ \cite[$\S$3]{KS21}}{}]\label{prop:genheight-cf}
Let \. $a_0,  \ldots, a_s \in \nz $\..
Then there exists a poset \ts $P=(X,\prec)$ \ts of width two
on \. $|X|=a_0+\ldots +a_s$ \. elements, and element \ts $x \in \min(P)$,
such that
	\[ \agr(P,x) \ = \  [a_0 \ts ; \ts a_1,\ldots,a_s].
    \]
\end{prop}

 \smallskip

%
%

\begin{cor}
\label{cor:genheight-cf}
Let \. $a_1,\ldots a_s \in \nz$.
Then there exists a width two poset \ts $P=(X,\prec)$ \ts
on \. $|X|=a_1+\ldots +a_s$ \. elements, and element \ts $x \in \min(P)$,
such that
\[ \frac{1}{\agr(P,x)} \ = \  [0\ts ; \ts a_1,\ldots,a_s].\]
\end{cor}


\begin{proof}  This follows from \. $[a_1; \ts a_2,\ldots, a_s] \ts = \ts  [0\ts ; \ts a_1,\ldots, a_s]^{-1}$.
\end{proof}

\smallskip

\begin{rem}
Proposition~\ref{prop:genheight-cf} was proved implicitly in \cite[$\S$3]{KS21}.
Unfortunately, the notation and applications in that paper are very different
from ours, so we chose to include a self-contained proof for completeness.
\end{rem}

\smallskip


We now present the proof of Proposition~\ref{prop:genheight-cf}, which uses the following corollary of Lemma~\ref{lem:KS-1} and Lemma~\ref{lem:KS-2}.

\smallskip

\begin{cor}\label{cor:KS}
	Let \ts $P=(X,\prec)$ \ts be a width two poset on \ts $n=|X|$ elements, let \ts $x\in \min(P)$,
and let \ts $a \in \nz$.
	Then there exists a width two poset \ts $Q=(Y,\prec')$ \ts and \ts $y \in \min(Q)$,
such that \ts $|Y|=n+a$ \ts and
	\[  \agr(Q,y) \ = \  a \. + \. \frac{1}{\agr(P,x)}.\]
\end{cor}

\begin{proof}
Use Lemma~\ref{lem:KS-1} once, and Lemma~\ref{lem:KS-2} \. $(a-1)$ \. times.
Also note that the operations used in Lemma~\ref{lem:KS-1} and Lemma~\ref{lem:KS-2} do not increase the width of the poset $Q$ if the input poset $P$ is not a chain.
\end{proof}

\smallskip

\begin{proof}[Proof of Proposition~\ref{prop:genheight-cf}]
	We use induction on \ts $s$.
	For \ts $s=0$, let \. $P := C_{a_0-1} + \{x\}$ \. be a disjoint sum of two chains, and
observe that  \. $\ag(P,x) = a_0$.
	
	Suppose the claim holds for \ts $s-1$, i.e.\ there exists a poset \ts $P_1$ \ts
on \. $n=a_1+\ldots+a_s$ \. elements and \ts
$x_1 \in \min(P_1)$,  such that \. $\ag(P_1,x_1) \ts = \ts  [a_1 \ts ; \ts a_2,\ldots, a_s]$,
	and with \. $|P_1|=a_1+\ldots+a_s$.
	By Corollary~\ref{cor:KS}, there exists a poset \ts $Q$ \ts on \ts $a_0+n$ \ts elements,
and \ts $x \in \min(P)$, such that
	\begin{align*}
		\ag(P,x) \, = \,  a_0 \, + \, \frac{1}{\ag(P_1,x_1)} \, = \,
a_0 \, + \, \frac{1}{[a_1 \ts ; \ts a_2,\ldots, a_s]} \, = \, [a_0 \ts ; \ts a_1,\ldots,a_s].
	\end{align*}
This completes the proof.
\end{proof}

\smallskip

\subsection{Number theoretic estimates} \label{ss:verify-NT}
For \ts $A\in \zz_{\geq 1}$ \ts and \ts $m \in [A]$,
consider the quotients in the continued fraction of \ts $m/A$ \ts and their sum:
\[\frac{m}{A} \, = \, [0 \ts ; \ts a_1(m),\ldots, a_s(m)] \quad \text{and} \quad S_A(m) \, := \, \sum_{i=1}^s \. a_i(m)\..\]
Note that every rational number can be represented by continued fractions in two ways (depending if the last quotient is strictly greater than $1$, or is equal to $1$), and   \. $S_{A}(m)$ \. are equal for both representations.
Also note that
\begin{equation}\label{eq:cf-irr}
	S_{A}(m) \, = \,  S_{A'}(m'), \qquad \text{where} \quad A' \. := \. \frac{A}{\gcd(A,m)}
\quad \text{and} \quad m'\. := \. \frac{m}{\gcd(A,m)}
\end{equation}
are normalized to be coprime integers.
The following technical result will also be used in the  proof of the
Verification Lemma~\ref{l:Verify-Quart}.

\smallskip

\begin{prop}\label{p:NTD}
	There exists a constant \ts $C>0$, such that for all coprime integers \ts $A,B$ \ts which satisfy \. $C < B < A < 2B$,
	there exists an integer \. $m:=m(A,B)$ \. such that \. $m<B$,
	\[ S_A(m) \ \leq \  2 \. (\log A)^2 \quad \text{ and } \quad   S_B(m) \ \leq \  2 \. (\log B)^2.  \]
\end{prop}

\smallskip


We now build toward the proof of this result.  
We need the following technical result.

\smallskip

\begin{lemma}[{\rm Yao--Knuth \cite{YK75}}{}]\label{lem:Knuth}
We have:
	\[  \frac{1}{n} \. \sum_{m \ts \in \ts [n]}  S_n(m) \ = \
\frac{6}{\pi^2} \. (\log n)^2 \. + \. O\big((\log n) (\log \log n)^2\big) \quad \text{as \ \  $n\to \infty$}.\]
\end{lemma}

\smallskip

By the Markov inequality, it follows from Lemma~\ref{lem:Knuth} that
\begin{equation}\label{eq:Knuth}
\begin{split}
  \big|\big\{ m \in [n] \, : \,  S_n(m) \. > \. 2 \. (\log n)^2 \.  \big\}\big|
& \  \leq  \  \frac{3}{\pi^2} \. n \. (1+o(1)).
\end{split}
\end{equation}

\smallskip

\begin{proof}[Proof of Proposition~\ref{p:NTD}]
	Denote
	\begin{align*}
		\vt(A,B) \, := \, \big|\{ m \in [B] \, : \,  S_A(m) \leq
2 \. (\log A)^2, \ S_B(m) \leq 2 \. (\log B)^2 \  \} \big| \ts.
	\end{align*}
To prove the result, it sufficed to show that
$$\vt(A,B) \, = \, \Omega\big(B\big) \quad \text{as \ \ $C\. \to \. \infty$\ts.}
$$

Now, it follows from the inclusion-exclusion principle, that
\[
\vt(A,B) \ \ge \ B \. - \.   \big|\big\{ m \in [B] \, : \,  S_A(m) >
2 \. (\log A)^2 \  \big\} \big| \. - \.  \big|\big\{ \. m \in [B]  \, : \,   \ S_B(m) > 2 \. (\log B)^2 \.  \big\} \big|\ts.
\]
On the other hand, we have:
\begin{align*}
\big|\{ m \in [B]  \, : \,  S_A(m) >
2 \. (\log A)^2 \  \} \big| \ \leq \  \big|\{ m \in [A]  \, : \,  S_A(m) >
2 \. (\log A)^2 \  \} \big| \ \leq_{\eqref{eq:Knuth}} \ \tfrac{3}{\pi^2} \. A \. (1+o(1)),
\end{align*}
and
\begin{align*}
	\big|\{ m \in [B]  \, : \,   S_B(m) >
	2 \. (\log B)^2 \  \} \big| \ \leq_{\eqref{eq:Knuth}} \ \tfrac{3}{\pi^2} \. B \. (1+o(1)).
\end{align*}
Combining these inequalities, we get
\begin{align*}
\vt(A,B) \ \geq \ B  \.  - \. \tfrac{3}{\pi^2} \. (A +B) \big(1+o(1)\big) \  \geq  \ B \big(1-\tfrac{9}{\pi^2} \big) \big(1-o(1)\big),
\end{align*}
and the result follows since \. $\big(1-\frac{9}{\pi^2}\big) > 0$.
\end{proof}


\begin{rem}
The proof of Proposition~\ref{p:NTD} does not give a (deterministic)
polynomial time algorithm to find the desired~$m$, i.e.\ in \ts poly$\ts (\log A)$ \ts time.
There is, however, a relatively simple \emph{probabilistic} \ts polynomial
time algorithm, cf.\ \cite[Rem.~5.31]{CP23}.  Most recently, we were able to improve
upon the estimate in Proposition~\ref{p:NTD} using Larcher's bound, see \cite[$\S$1.5]i{CP-CF}. 
\end{rem}

\smallskip

\subsection{Bounds on relative numbers of linear extensions}\label{ss:verify-relative-bounds}
The following simple bound is the final ingredient we need for the proof of the Verification Lemma.

\smallskip

\begin{prop}[{\rm see \cite{CPP-Quant-CPC,EHS}}{}]\label{p:g-bound}
Let \ts $P=(X,\prec)$ \ts be a poset on \ts $|X|=n$ \ts elements, and let \ts $x\in \min(X)$.
Then \. $1 \leq \ag(P,x) \leq n$.  Moreover, \ts $\ag(P,x)=1$ \ts if an only if \ts $\min(P)=\{x\}$,
i.e.\ \ts $x$ \ts is the
unique minimal element.
\end{prop}

\smallskip

The lower bound holds for all \ts $x\in X$, see e.g.\ \cite{EHS}.
The upper bound is a special case of \ts \cite[Lem.~5.1]{CPP-Quant-CPC}.
We include a short proof for completeness.

\smallskip

\begin{proof}
The lower bound \. $e(P-x) \le e(P)$ \. follows from the injection \. $\Ec(P-x) \to \Ec(P)$ \.
that maps \ts $f\in \Ec(P-x)$ \ts into \ts $g\in \Ec(P)$ \ts by letting \ts $g(x)\gets 1$, \. $g(y)\gets f(x)+1$ \.
for all \ts $y\ne x$.  For the second part, note that \. $e(P)-e(P-x)$ \. is the number of \ts $f\in \Ec(P)$ \ts
such that \ts $f(x) >1$, so \. $e(P)-e(P-x)=0$ \. implies \.  $\min(P)=\{x\}$.

The upper bound \. $e(P) \le n \ts e(P-x)$ \.
follows from the injection \. $\Ec(P) \to  \Ec(P-x) \times [n] $ \.
that maps \ts $g\in \Ec(P)$ \ts into a pair \ts $\big(f, g(x)\big)$ \ts where \ts $f\in \Ec(P-x)$ \ts
is defined as \ts $f(y)\gets g(y)$ \ts if \ts $g(y)<g(x)$, \. $f(y)\gets g(y)-1$ \ts if \ts $g(y)>g(x)$.
\end{proof}

\smallskip

\subsection{Proof of Verification Lemma~\ref{l:Verify-Quart}}\label{ss:verify-proof}
Recall the decision problem
\begin{equation*}
	\text{\sc VerRLE} \ := \ \big\{ \.	
	\ag(P,x) \.  =^? \,
	\tfrac{A}{B}
	\. \big\},
\end{equation*}
where \ts $P=(X,\prec)$ \ts is a poset on \ts $n=|X|$ \ts elements, \ts $x\in \min(P)$,
and \ts $A,B$ are coprime integers with \ts $B< A\leq n!$\..
We simulate \. {\sc VerRLE} \. with an oracle for \. {\sc QuadRLE} \. as follows.

By Proposition~\ref{p:g-bound},  we need only to consider the cases \. $1 < \frac{A}{B} \leq n$.
Indeed, when \ts $\ag(P,x)<1$ \ts or \ts $\ag(P,x)>n!$, \. {\sc VerRLE} \ts does not hold.
Additionally, when \ts $\ag(P,x)=1$, \. {\sc VerRLE} \ts holds if and only if \ts $P$ \ts is a chain.
Let \ts $k :=  \left\lfloor \tfrac{A}{B} \right \rfloor$.
As in the $s=0$ part of the proof of Proposition~\ref{prop:genheight-cf},
there exists a poset \ts $P_3=(X_3,\prec_3)$ \ts with $|X_3| = k \le n$,
and an element \ts $x_3 \in \min(P_3)$, such that \. $\ag(P_3,x_3) = k$.


Let $A',B'$ be coprime integers such that
$$\frac{A}{B} \, =  \, k \. \frac{A'}{B'}\,.
$$
Then we have \. $B\le B'<A'<2B'$, \. $A'\le A$ \. and thus \. $\log A' = O(n \log n)$.
By Proposition~\ref{p:NTD},
there is a positive integer \. $m \in [B']$, such that
\[ S_{A'}(m) \ \leq \  2 \. (\log A')^2 \. \quad \text{ and } \quad   S_{B'}(m) \ \leq \  2 \. (\log B')^2 .  \]
At this point we \emph{guess} \ts such~$m$.  Since computing the quotients of \ts $m/A'$ \ts
can be done in polynomial time, we can verify in polynomial time that \ts $m$ \ts satisfies
the inequalities above.

By Corollary~\ref{cor:genheight-cf}, we can construct posets \ts $P_2=(X_2,\prec_2)$,
\ts $P_4=(X_4,\prec_4)$ \ts with \ts
$x_2\in \min(P_2)$, \ts $x_4\in \min(P_2)$,
such that
 \[
 \ag(P_2,x_2) \, = \,  \frac{B'}{m} \qquad \text{and} \qquad \ag(P_4,x_4) \, = \, \frac{A'}{m} \,.
 \]
The corollary also gives us
 \[ |X_2|  \, \leq \, S_{B'}(m) \, \leq \, 2 (\log B')^2 \.  \, = \, O\big(n^2 (\log n)^2\big),
 \]
and we similarly have \. $|X_4| = O\big(n^2 (\log n)^2\big)$.
Since posets \ts $P_2,P_3$ \ts and \ts $P_4$ \ts have polynomial size, we can call \ts {\sc QuadRLE} \ts
to check
\begin{equation*}
	\big\{ \.	
	\ag(P,x) \cdot  \ag(P_2,x_2) \, =^? \,
	\ag(P_3,x_3) \cdot  \ag(P_4,x_4)
	\. \big\}.
\end{equation*}
Observe that
$$  \frac{\ag(P_3,x_3) \cdot \ag(P_4,x_4)}{\ag(P_2,x_2)} \ = \  \frac{m}{B'} \. \cdot \.  k \. \cdot \. \frac{A'}{m} \ = \  \frac{A}{B}\,.
$$
Thus, in this case \ts {\sc QuadRLE} \ts is equivalent to \ts {\sc VerRLE}, as desired. \qed

\smallskip

\begin{rem} \label{r:CF}
In our recent paper \cite{CP24-CF}, we use ideas from the proof above
to obtain further results for relative numbers of linear extensions.
We also use stronger number theoretic estimates than those given by
Lemma~\ref{lem:Knuth}.
\end{rem}



\medskip

\section{Fixing one element}\label{s:Sta1}

In this section we prove Theorem~\ref{t:ESta-1}.  The proof relies heavily
on~\cite{MS22}.  We also need the definition and basic properties of the
\defng{promotion} \ts and \defng{demotion} \ts operations on
linear extensions, see e.g.\ \cite{Sta-promo} and \cite[$\S$3.20]{Sta-EC}.

\subsection{Explicit equality conditions} \label{ss:Sta1-equality}
For \ts $k=1$, the equality cases of Stanley's
inequality~\eqref{eq:Sta} are tuples \. $(P,x,z,a,c)$, \ts where
\ts $P=(X,\prec)$ \ts is a poset on \ts $n=|X|$ \ts elements,
\ts $x,z\in X$,  \ts $a,c \in [n]$, and  the following holds:
\begin{equation}\label{eq:Sta1}
\aN_{z\ts c}(P, x,a)^2 \, = \, \aN_{z\ts c}(P, x,a+1)\cdot  \aN_{z\ts c}(P, x,a-1).
\end{equation}
The subscripts here and throughout
this section are no longer bold, to emphasize that \ts $k=1$.
Recall also both the notation in~$\S$\ref{ss:intro-Stanley}, and the Ma--Shenfeld
poset notation in~$\S$\ref{ss:main-MS}.

\smallskip

\begin{lemma}\label{lem:k=1}
Let \ts $P=(X,\prec)$ \ts be a poset on \ts $n=|X|$ \ts elements, let \ts $x,z\in X$, \ts $a,c\in [n]$.
Then the equality \ts \eqref{eq:Sta1} \ts is equivalent to:

\smallskip

\qquad
$(\divideontimes)$ \quad for every \ts $f \in \Ec_{z\ts c}(P,x,a+s)$, \ts  $s \in \{0,\pm 1\}$,
we have \. $x \. \| \, \lc(f)$  \. and \.  $x \. \| \, \uc(f)$.	
\end{lemma}

\smallskip

We prove Lemma~\ref{lem:k=1} later in this section.

\smallskip

\begin{rem}\label{r:Sta1-MS}
For the case~$k=0$, the analogue of $(\divideontimes)$ 
that companions of $f$ are incomparable to $x$, was proved
in~\cite[Thm~15.3(c)]{SvH-acta}.  However, $(\divideontimes)$ fails for \ts $k= 2$,
as shown in the ``hope shattered'' Example~1.4 in~\cite{MS22}.
Thus, Lemma~\ref{lem:k=1} closes the gap between these two results.
See~$\S$\ref{ss:finrem-LE} for potential complexity implications of this observation.

Note also that condition $(\divideontimes)$ is in~$\ts \P$ \ts since can be
equivalently described in terms of explicit conditions on the partial order
(rather than in terms of linear extensions of the poset). This is proved
in \cite[Thm~15.3{}\ts{}(d)]{SvH-acta} for \ts $k=0$, and in \cite[Eq.~(1.6)]{MS22}
for \ts $k=1$.
\end{rem}

\smallskip

\begin{proof}[{Proof of Theorem~\ref{t:ESta-1}}]
As before, let \ts $P=(X,\prec)$ \ts be a poset on \ts $n=|X|$ \ts elements,
let \ts $x,y,z\in X$ \ts and \ts $a,b,c\in [n]$.  Denote by \.
$\aN_{z c}(P,x,a,y,b)$ \. the number of linear extensions \.
$f \in \Ec_{z \ts c}(P,x,a)$ \ts that additionally satisfy \ts $f(y)=b$.

Now, condition~$(\divideontimes)$ in  Lemma~\ref{lem:k=1}, can be rewritten as follows:
		\begin{equation}\label{eq:sapporo-1}
		 \aN_{z \ts c}(P,x,a',y,b') \. = \. 0  \quad \text{for all} \quad y \in \cC(x) \ \ \ \text{and} \ \ \  a',b' \in \{a-1,a,a+1\}.
		\end{equation}
Indeed, each vanishing condition in \eqref{eq:sapporo-1} is checking whether there exists a companion of~$x$
in a linear extension that is comparable to~$x$.

Recall that each vanishing condition in \eqref{eq:sapporo-1} is in~$\P$, see references in~$\S$\ref{ss:hist-Stanley}.
There are at most \ts $6n$ \ts instances to check, since for all \ts $y\in X$ \ts there are at most $6$ choices of distinct
\ts $a',b'$ \ts in \ts $\{a-1,a,a+1\}$.  Therefore, \ts {\sc EqualityStanley}$_{1} \in \P$.
\end{proof}

\smallskip

\subsection{Ma--Shenfeld theory}\label{ss:Sta1-critical}
We now  present several ingredients needed to prove Lemma~\ref{lem:k=1}.
We follow closely the Ma--Shenfeld paper~\cite{MS22}, presenting several
results from that paper.

\smallskip

In \cite{MS22}, Ma--Shenfeld defined the notions of \defn{subcritical},
\defn{critical}, and \defn{supercritical} \ts posets, which are directly
analogous to the corresponding notions for polytopes given in \cite{SvH-acta},
cf.~$\S$\ref{ss:hist-AF}.
As the precise definitions are rather technical,  we will not state
them here while still including key properties of those families that
are needed to prove Lemma~\ref{lem:k=1}.

We start with the following hierarchical relationship between the three families:
\[ \{\text{subcritical posets}\} \quad \supseteq \quad  \{\text{critical posets}\}  \quad \supseteq \quad \{\text{supercritical posets}\}.\]
A poset that is subcritical but not critical is called \defnb{sharp subcritical}, and a poset that is critical but not super critical is called \defnb{sharp critical}.

The equality conditions for \eqref{eq:Sta1} are directly determined by the classes to which the poset $P$ belongs, as we explain below.
We note that these families depend on the choices of \ts $P,x,a,z,c$, which we omit from the notation to improve readability.
Furthermore, without loss of generality we can assume that \. $z \notin \{a-1,a,a+1\}$,
as otherwise one of the numbers in \eqref{eq:Sta1} are equal to $0$, making
the problem in~$\P$ (see above).

\smallskip

We now state two other properties of these families, which  require the following definitions.
Following \cite{MS22}, we  add two
elements \. $z_0,z_{k+1}$ \. into the poset such that \. $z_0 \preccurlyeq y \preccurlyeq z_{k+1}$ \.
for all \ts $y \in X$, and we define \. $c_0:=0$ \. and \. $c_{k+1}:=n+1$\..
A \defnb{splitting pair} \ts is a pair of integers \ts $(r,s)$ \ts in \ts  $\{0,\ldots, k+1\}$,
such that \. $(r,s)\neq (0,k+1)$\..\footnote{In~\cite[Def~5.2]{MS22},
this pair is instead written as $(r+1,s)$.}

 \smallskip

\begin{lemma}[{\cite[Lemma~5.10]{MS22}}]\label{lem:subcrit-prop}
	Let \ts $P=(X,\prec)$ \ts be a  sharp subcritical poset.
	Then there exists a splitting pair \ts $(r,s)$ \ts such that
	\begin{equation}\label{eq:split}
		 \big|\big\{ \ts u \in X \, : \, z_r \ts \prec \ts u \ts \prec \ts  z_s \ts \big\}\big| \, = \,  c_s\. - \. c_r \. - \. 1.
	\end{equation}\end{lemma}
\smallskip

We say that poset \ts $P$ \ts is \defnb{split indecomposable} \ts if, for every splitting pair $(r,s)$,
\[   		 \big|\big\{ \ts u \in X \, : \, z_r \ts \prec \ts u \ts \prec \ts  z_s \ts \big\}\big| \, \leq \,  c_s\. - \. c_r \. - \. 2. \]
In particular, by Lemma~\ref{lem:subcrit-prop}
every sharp subcritical poset is not split indecomposable.

It was shown in \cite{MS22}, that we can without loss of generality assume
that poset \ts $P$ \ts is split indecomposable.  Indeed, otherwise checking
\eqref{eq:Sta1} can be reduced to checking the same problem for a
smaller poset:  either restricting to the set in \eqref{eq:split},
or removing this set from the poset, see \cite[$\S$6]{MS22} for details.
Thus we can without loss of generality assume that \ts $P$ \ts is a critical poset.
\smallskip

\begin{lemma}
[{\cite[Lemma~5.11]{MS22}}]\label{lem:crit-prop}
Let \ts $P$ \ts be a split indecomposable sharp critical poset.
Then there exists a splitting pair \ts $(r,s)$ \ts such that \. $c_r<a<c_s$ \. and
\begin{equation}\label{eq:split-2}
		 \big|\big\{ \ts u \in X  \, : \, z_r \ts \prec \ts u \ts \prec \ts
z_s \ts \big\}\big| \, = \,  c_s\. - \. c_r \. - \. 2.
\end{equation}
\end{lemma}

\smallskip

\begin{rem}\label{r:MS-critical}
Lemmas~\ref{lem:subcrit-prop} and~\ref{lem:crit-prop} can be modified to imply
that deciding whether poset \ts $P$ \ts is subcritical, critical, or supercritical
is in~$\P$.  We do not need this result for the proof of Lemma~\ref{lem:k=1},
so we omit these changes to stay close to the presentation in~\cite{MS22}.
More generally, one can ask similar questions for H-polytopes 
(i.e., deciding if a given collection of polytopes is subcritical/critical/supercritical).  
While we believe that for TU-polytopes these decision problems are still likely
to be in \ts $\P$, proving that would already be an interesting challenge
beyond the scope of this paper.
\end{rem}

\smallskip

%
Recall from  \S\ref{s:Flat-Sta-proof}  that \. $\cF(a,\comp, \comp)$ \.
is the set of linear extensions in \. $\Ec_{z \ts c}(P,x,a)$, such that
both the lower and upper companions of $x$ are incomparable to~$x$.
Next,
 \. $\cF(a,\comp, \inc)$ \. is the set of linear extensions in \. $\Ec_{z \ts c}(P,x,a)$,
 such that  the lower companion is comparable to~$x$, but the upper companion is incomparable to~$x$.
Similarly,
\. $\cF(a,\inc, \comp)$ \. is the set of linear extensions in \. $\Ec_{z \ts c}(P,x,a)$,
such that  the lower companion is incomparable to~$x$, but the upper companion is comparable to~$x$.
Let \. $\cF(a-1,\cdot, \cdot)$ \. and \. $\cF(a+1,\cdot, \cdot)$ \. be defined analogously.
Finally, let \. $\aF(a+s,\cdot,\cdot) \. := \. |\cF(a+s,\cdot,\cdot)|$ \. where \. $s\in \{0,\pm 1\}$,
be the numbers of these linear extensions.

\smallskip

\begin{lemma}[{\cite[Thm~1.5]{MS22}}]\label{lem:crit-cond}
	Let \ts $P$ \ts be a  critical poset.
	 Then \ts \eqref{eq:Sta1} \ts holds \.
	\underline{if and only if}
	\begin{align}
& \aFr(a-1,\textnormal{com},\textnormal{com})\ = \ \aFr(a,\textnormal{com},\textnormal{com}) \ = \ \aFr(a+1,\textnormal{com},\textnormal{com}) \, = \, 0 \quad \ \text{ and} \label{eq:crit-1}\\
& \aligned & \aFr(a-1,\textnormal{com},\textnormal{inc}) \ = \ \aFr(a-1,\textnormal{inc},\textnormal{com})
\ = \ \aFr(a,\textnormal{com},\textnormal{inc})  \\
& \hskip1.cm = \ \aFr(a,\textnormal{inc},\textnormal{com})
\ = \ \aFr(a+1,\textnormal{com},\textnormal{inc}) \ = \ \aFr(a+1,\textnormal{inc},\textnormal{com}).
\endaligned \label{eq:crit-2}
	\end{align}
\end{lemma}

\smallskip

Now note that \. $\aF(a-1,\comp,\inc) \ \leq \ \aF(a-1,\inc,\comp)$,
with the equality \. \underline{if and only if} \. every upper companion of~$x$
is always incomparable to the lower companion of~$x$.  By an analogous arguments
applied to \. $\aF(a,\cdot, \cdot)$ \. and \. $\aF(a+1,\cdot, \cdot)$, we get
the following corollary.

\smallskip

\begin{cor}\label{cor:crit-cond}
	Let \ts $P$ \ts be a  critical poset.  Suppose
	\[ \aNr_{z \ts c}(P, x,a)^2 \, = \, \aNr_{z \ts c}(P, x,a+1) \.\cdot \.  \aNr_{z \ts c}(P, x,a-1) \ \neq  \ 0, \]
	Then, for every linear extension \ts $f\in \Ec(P)$ \ts counted by \eqref{eq:crit-2},
	the upper companion is incomparable to the lower companion: \. $\uc(f) \. \| \, {} \lc(f)$.
\end{cor}

\smallskip

Finally, we have equality conditions for supercritical posets.

\smallskip

\begin{lemma}
	[{\cite[Thm~1.3]{MS22}}]\label{lem:supercrit-cond}
	Let $P$ be a supercritical poset.
	Then \ts \eqref{eq:Sta1} \ts holds \.
	\underline{if and only if} \. equalities \ts \eqref{eq:crit-1} \ts and \ts \eqref{eq:crit-2} \ts hold, and additionally
	\begin{equation}\label{eq:crit-3}
\text{all numbers in \eqref{eq:crit-2} are equal to~$\ts 0$.}
	\end{equation}
\end{lemma}

\smallskip

\subsection{Proof of Lemma~\ref{lem:k=1}}\label{ss:Sta1-proofi}
%
Note that \ts \eqref{eq:crit-1}, \ts \eqref{eq:crit-2} \ts and \ts \eqref{eq:crit-3}
are equivalent to requiring that \ts $x$ \ts is incomparable to both \. $\lc(f)$  \. and \.  $\uc(f)$.
Thus it suffices to show that, if \ts $P$ \ts is a critical poset, then \eqref{eq:crit-3} holds.

Suppose to the contrary, that \ts $P=(X,\prec)$ \ts is a counterexample,
and let \ts $n:=|X|$.  Then \ts $P$ \ts is a sharp critical poset.
By taking the dual poset if necessary, we can assume, without loss of generality,
that \ts $c<a$.  It then follows that the splitting pair \ts $(r,s)$ \ts in
Lemma~\ref{lem:crit-prop} is \ts $(1,2)$.
This means that \ts $c_r=c$ \ts and \ts $c_s=n+1$, so we have from \eqref{eq:split-2} that
	\begin{equation}\label{eq:split-3}
			\left|\{ u \in P \, :  \, z \. \prec \. u  \}\right| \ = \  n -c-1.
	\end{equation}
	
	Since \eqref{eq:crit-3} does not hold, there exist \. $f \in  \cF(a,\comp, \inc)$ \. and \. $h \in \cF(a-1,\comp,\inc)$\..
	Let \. $y_1:=f^{-1}(a-1)$ \. (i.e., the lower companion  in $f$)
	and \. $y_2:=h^{-1}(a)$ \. (i.e., the lower companion  in $h$).
	Note that we have \. $y_1 \. \prec \. x \. \prec \. y_2$\..
	Let \. $m=f(y_2)$\., and note that \. $m\geq a+2$ \. by definition.

\smallskip

\nin
{\bf We claim:} \. There exists a new linear extension \. $g\in \Ec(P)$ \. such that $g(y_2)=m-1$, and
	such that \. $g \in \cF(a,\comp,\inc)$ \. if \ts $m>a+2$,
	and \. $g \in \cF(a,\comp,\comp)$  \. if \ts $m=a+2$.
	Note that this suffices to prove the lemma, as by replacing $f$ with $g$ and decreasing $m$ repeatedly,
	we get that
\. $\aF(a,\comp,\comp)>0$, which contradicts \eqref{eq:crit-1}.

\smallskip
	
\nin
\textbf{\em We now prove the claim.}
	Since $h(y_2)=a <m=f(y_2)$, there exists \. $w \in X$ \. such that  \. $f(w) < m$ \.
	and \. $w \. \| \. y_2$\ts.
	Suppose \ts $w$ \ts is such an element that maximizes \ts $f(w)$.
	There are  \emph{three cases}:
	
\smallskip

	\underline{First}\ts, \ts suppose that \. $f(w)> a$.  By the maximality assumption, every element ordered
    between $w$ and $y_2$ according to~$f$, is incomparable to~$w$.
	Then we can promote \ts $w$ \ts to be larger than~$\ts y_2$.
	Note that the resulting linear extension \ts $g\in \Ec(P)$ \ts satisfies \. $g(y_2)=m-1$,
    \. $g(y_1)=a-1$ \. and  \. $g(x)=a$, as desired.

\smallskip
	
	\underline{Second}\ts, \ts suppose that \. $c<f(w)<a$.
	By the maximality assumption, every element ordered between \ts $w$ \ts and \ts $y_2$
    according to~$f$, is incomparable to~$\ts w$.
	Then we can promote \ts $w$ \ts to be larger than~$\ts y_2$.
	The resulting linear extension \ts $g'\in \Ec(P)$ \ts satisfies \. $g'(y_2)=m-1$.
	Note, however, that we have \. $g'(y_1)=a-2$ \.  and \. $g'(x)=a-1$.
	In order to fix this, let \. $v:=f^{-1}(a+1)$.
	It follows from Corollary~\ref{cor:crit-cond}, that $v$ is incomparable to~$y_1$ and~$x$.
    Thus we can demote $v$ to be the smaller than~$y_1$.  We obtain a new linear extension
    \ts $g\in \Ec(P)$ \ts
    that satisfies \. $g(y_1)=a-1$ \. and  \. $g(x)=a$, as desired.

\smallskip
	
	\underline{Third}\ts, \ts suppose that \. $f(w)<c$.
	Then, every element ordered between $z$ and $y_2$ according to~$\ts f$,
is less than~$\ts y_2 \ts.$  Note that there are   \. $m-c-1$ \. many such elements.
	On the other hand, it follows from \eqref{eq:split-3}, that there is exactly  one element in
\. $\{f^{-1}(c+1), f^{-1}(c+2), \ldots, f^{-1}(n) \}$ \.
	that is incomparable to~$\ts z$.
	It then follows that there are at least \. $m-c-2$ \.
	elements that are greater than \ts $z$ \ts and less than~$\ts y_2$\ts, i.e.\
		\begin{equation}
		\big|\big\{ u \in X \, : \, z \. \prec \. u \. \prec \.  y_2 \big\}\big| \, \geq  \,  m-c-2.
\end{equation}
On the other hand, the existence of \ts $h$ \ts implies that
		\begin{equation}
	\big|\big\{ u \in X \, : \, z \. \prec \. u \. \prec \.  y_2 \big\}\big| \,
\leq   \,  h(y_2)-c-1 \, = \, a-c-1 \, \leq \, m-c-3,
\end{equation}
a contradiction. This finishes the proof of the claim. \qed



\medskip

\section{Final remarks}\label{s:finrem}

{\small

\subsection{The basis of our work}\label{ss:finrem-hist}
Due to the multidisciplinary nature of this paper, we make a special effort to
simplify the presentation.  Namely, the proofs of our main results
(Theorems~\ref{t:main-AF} and~\ref{t:main-Sta}), are largely
self-contained in a sense that we only use standard results
in combinatorics (Stanley's theorem in~$\S$\ref{ss:AF-proofs} and the
Brightwell--Winkler's theorem in~$\S$\ref{ss:hist-CC}),
computational complexity (Toda's theorem in~$\S$\ref{ss:roadmap-proofs}),
and number theory (Yao--Knuth's theorem in~$\S$\ref{ss:verify-NT}).  In reality,
the paper freely uses tools and ideas from several recent
results worth acknowledging.

First, we heavily build on the recent paper by Shenfeld and van Handel
\cite{SvH-acta}, and the followup by Ma and Shenfeld~\cite{MS22}.
Without these results we would not know where to look for ``bad posets''
and ``bad polytopes''.  Additionally, the proof in~$\S$\ref{s:Flat-Sta-proof}
is a reworking and simplification of many technical results and ideas
in~\cite{MS22}.

Second, in~$\S$\ref{ss:verify-CF} we use and largely rework
the continued fraction approach
by Kravitz and Sah \cite{KS21}.  There, the authors employ the
\defng{Stern--Brocot} \ts and \defng{Calkin--Wilf tree} \ts notions, which
we avoid in our presentation as we aim for different applications.

Third, in the heart of our proof of Theorem~\ref{t:main-Sta}
in~$\S$\ref{ss:roadmap-proofs}, we follow the complexity roadmap
championed by Ikenmeyer, Panova and the second author in \cite{IP22,IPP22}.
Same for the heart of the proof of the Verification Lemma~\ref{l:Verify-Quart}
in~$\S$\ref{ss:verify-proof},
which follows the approach in our companion  paper~\cite{CP23}.

\smallskip

On the other hand, the proof of Theorem~\ref{t:ESta-1} given in Section~\ref{s:Sta1},
is the opposite of self-contained, as we rely heavily on both results and ideas
in~\cite{MS22}.  We also use properties the \defng{promotion} \ts and \defng{demotion} \ts
operations on linear extensions, that were introduced by Sch\"utzenberger in the
context of algebraic combinatorics, see \cite{Schu}\footnote{These operations
were rediscovered in \cite{DD,Day84}, where they are called \defng{push up} \ts
and \defng{push down}, respectively.}.   Panova and the authors employed
this approach in a closely related setting in \cite{CPP-KS,CPP-effective,CPP-Quant-CPC}.
We emphasize once again that our proof of Theorem~\ref{t:ESta-1} is independent
of the rest of the paper and is the only part that uses results in~\cite{MS22}.

\subsection{Equality cases}\label{ss:finrem-equality}
The reader unfamiliar with the subject may wonder whether equality
conditions of known inequalities are worth an extensive investigation.
Here is how Gardner addresses this question:

\smallskip

\begin{center}\begin{minipage}{13.cm}%
{{\em ``If inequalities are silver currency in mathematics, those that come along
with precise equality conditions are gold. Equality conditions are treasure boxes
containing valuable information.''} \cite[p.~360]{Gar02}.
}
\end{minipage}\end{center}

\smallskip

\nin
Closer to the subject of this paper, Shenfeld and van Handel explain the difficulty
of finding equality conditions for \eqref{eq:MQI} and~\eqref{eq:AF}:

\smallskip

\begin{center}\begin{minipage}{13.cm}%
{{\em ``In first instance, it may be expected that the characterization of the extremals of the
Minkowski and Alexandrov--Fenchel inequalities should follow from a careful analysis of
the proofs of these inequalities. It turns out, however, that none of the classical
proofs provides information on the cases of equality: the proofs rely on strong
regularity assumptions (such as smooth bodies or polytopes with restricted face directions)
under which only trivial equality cases arise, and deduce the general result by
approximation. The study of the nontrivial extremals requires one to work directly
with general convex bodies, whose analysis gives rise to basic open questions in the
foundation of convex geometry.''} \cite[p.~962]{SvH-duke}.
}
\end{minipage}\end{center}

\smallskip

\subsection{Polytopes}\label{ss:finrem-polytopes}
The family of TU-polytopes that we chose is very special in that these H-polytopes
have integral vertices (but not a description in~$\P$, as V-polytopes are defined
to have).  In \cite{CP23+}, we consider a family of \defng{axis-parallel boxes} \ts
which have similar properties.
Clearly, for general convex bodies there is no natural way to set up a
computational problem that would not be immediately intractable
(unless one moves to a more powerful computational model, see e.g.~\cite{BCSS98}).

\subsection{Discrete isoperimetric inequality} \label{ss:finrem-discrete-isop}
For a discrete version of the isoperimetric inequality in the plane, one can consider
convex polygons with given normals to edges.  In this case, L'Huilier (1775)
proved that the isoperimetric ratio is minimized for circumscribed polygons,
see e.g.\ \cite[$\S$I.4]{Fej}.
%
In the~1860s, Steiner and Lindel\"of studied a natural generalization of
this problem in~$\rr^3$, but were unable to solve it in full generality.

At the turn of 20th century, Minkowski developed the
\defng{theory of mixed volumes}, motivated in part to
resolve the Steiner--Lindel\"of problem.  He showed that among all polytopes with
given normals, the isoperimetric ratio is minimized on circumscribed polytopes,
see e.g.\ \cite[$\S$V.7]{Fej}.

There are several Bonnesen type and stability versions of the
discrete isoperimetric inequality, see e.g.\ \cite{FRS,IN15,Zhang98}.
Let us single out a hexagon version used by Hales in his famous proof
of the \defng{honeycomb conjecture} \cite[Thm~4]{Hales}.
%


\subsection{Brunn--Minkowski inequality} \label{ss:finrem-BM}
There are several proofs of the Brunn--Minkowski inequality~\eqref{eq:BM},
but some of them do not imply the equality conditions, such as, e.g., the ``brick-by-brick''
inductive argument in \cite[$\S$12.2]{Mat}.  Note also that Alexandrov's proof of the
\defng{Minkowski uniqueness theorem} \ts (of polytopes with given facet volumes and normals)
relies on the equality conditions for the Brunn--Minkowski inequality, see \cite{Ale-book}.
This is essential for Alexandrov's ``topological method'', and is the basis for the
\defng{variational principle} \ts approach, see e.g.\ \cite{Pak-book}.

\subsection{Van der Waerden conjecture} \label{ss:finrem-vdW}
The Alexandrov--Fenchel inequality~\eqref{eq:AF} came to prominence
in combinatorics after Egorychev~\cite{Egor} used it to prove the
\defn{van der Waerden conjecture}, that was proved earlier
by Falikman~\cite{Fal}.\footnote{According to
Vladimir Gurvich's \href{https://tinyurl.com/47c6s9et}{essay},
Egorychev was the referee of Falikman's article which
was submitted prior to Egorychev's preprint.
}
See~\cite{Knuth81,vL} for friendly expositions.
This development set the stage for Stanley's paper~\cite{Sta-AF}.
The conjecture states that for every bistochastic $n\times n$
matrix~$A$, we have
\begin{equation}\label{eq:vdW}\tag{vdW}
\per(A) \, \ge \, \frac{n!}{n^n}\.,
\end{equation}
and the equality holds only if \ts $A=(a_{ij})$  \ts has uniform entries: \ts $a_{ij}=\frac{1}{n}$ \ts for all $1\le i,j\le n$.

Note that Egorychev's proof of the equality conditions for~\eqref{eq:vdW}
used Alexandrov's equality conditions~\eqref{eq:AF} for
nondegenerate boxes, see~$\S$\ref{ss:hist-AF} (cf.~\cite[p.~735]{Knuth81} and \cite[$\S$7]{vL}).
In a followup paper~\cite{CP23+}, we analyze the complexity of the
Alexandrov--Fenchel equality condition for degenerate boxes.
Note also that Knuth's exposition in \cite{Knuth81} is essentially
self-contained, while Gurvits's proof of~\eqref{eq:vdW}
completely avoids~\eqref{eq:AF}, see~\cite{Gur,LS10}.

\subsection{Matroid inequalities}\label{ss:finrem-matroid}
Of the several log-concavity applications of the AF inequality given
by Stanley in~\cite{Sta-AF} (see also~\cite[$\S$6]{Sta-two}), one stands out
as a special case of a Mason's conjecture (Thm~2.9 in~\cite{Sta-AF}).
The strongest of the three Mason's conjectures states that the numbers
\. $\rI(M,k)/\binom{n}{k}$ \. are log-concave,
where \ts $\rI(M,k)$ \ts is the number of independent sets of size~$k$
in a matroid~$M$ on~$n$ elements.
These Mason's conjectures were recently proved in a long series
of spectacular papers culminating with \cite{AHK,ALOV,BH20},
see also an overview in~\cite{Huh,Kalai}.

Curiously, the equality cases for these inequalities are
rather trivial and can be verified in polynomial time \cite{MNY}
(see also~\cite[$\S$1.6]{CP}).  Here we assume that the matroid is
given in a concise presentation (such presentations include graphical,
bicircular and representable matroids).  Curiously, for the weighted extension
of Mason's third conjecture given in \cite[Thm~1.6]{CP}, the equality cases
are more involved. It follows from \cite[Thm~1.9]{CP}, however, that this
problem is in~$\coNP$.  In other words, Theorem~\ref{t:main-Sta}
shows that \ts $\ESta_2$ \ts is likely much more powerful.

Note that the defect \. $\psi(M,k):=\rI(M,k)^2-\rI(M,k+1)\cdot\rI(M,k-1)$ \.
is conjectured to be not in \ts $\SP$, see \cite[Conj.~5.3]{Pak-OPAC}.
Clearly, the argument in the proof of Corollary~\ref{c:main-Stanley-not-SP}
does not apply in this case.  Thus, another approach is needed to prove this conjecture,
just as another approach is need to prove that \ts $\phi_0 \notin \SP$ \ts (see~$\S$\ref{ss:intro-Stanley}).

\subsection{Complexity of equality cases}\label{ss:finrem-LE}
Recall that
Theorem~\ref{t:main-AF} does not imply that \ts $\EAF$ \ts
is \ts $\NP$-hard or \ts $\coNP$-hard, more traditional
measures of computational hardness. This remains out of reach.
Note, however, that \ts $\ESta_k$ \ts is naturally in the class
\ts $\CEP$, see~$\S$\ref{ss:def-CS}.


\begin{conj}\label{conj:Sta-CEP}
\. $\ESta_k$ \ts is \ts $\CEP$-complete for large enough~$k$.
\end{conj}

\smallskip

If this holds for all \ts $k\ge 2$, this would imply a
remarkable dichotomy with \ts $k \le 1$ (see Theorem~\ref{t:ESta-1}).
To motivate the conjecture, recall from~$\S$\ref{ss:hist-CC},
that \ts $\CEP$-complete problem \ts $\text{\sc C}_\text{\#3SAT}$ \ts
is \ts $\coNP$-hard.  See \cite{CP23} for more on the complexity
of combinatorial coincidence problems.

Note that the proof of \ts $\ESta_2\notin\PH$ \ts implies
that \ts $\EAF \notin \PH$ \ts even when at most four polytopes are
allowed to be distinct.  It would be interesting to decide if this
number can be reduced down to three.  It is known that two distinct
TU-polytopes are not enough. This follows from a combination
of our argument that for supercritical cases (in the sense of \cite{SvH-acta}),
we have \ts $\EAF\in \coNP$, and an argument that for two polytopes the
equality cases are supercritical.\footnote{Ramon van Handel, personal
communication, April 2023.}

\subsection{Injective proofs}\label{ss:finrem-injective}
In enumerative combinatorics, whenever one has an equality
between the numbers counting certain combinatorial objects, one is tempted to
find a \defng{direct bijection} \ts between the sides, see e.g.\ \cite{Loe11,Pak-part,Sta-EC}.
Similarly, when presented an inequality \ts $f \geqslant g$, one is tempted to
find a \defng{direct injection}, see e.g.\ \cite{Pak,Sta-log-concave}.
In the context of linear extensions, such injections appear throughout
the literature, see e.g.\ \cite{Bre,BT02,CPP-KS,DD,GG22,LP07}.

Typically, a direct injection and its inverse are given by simple
polynomial time algorithms, thus giving a combinatorial interpretation
for the defect \ts $(f-g)$.  Therefore, if a combinatorial inequality
is not in~$\ts\SP$, it is very unlikely that there is a proof by a direct injection.
In particular, Corollary~\ref{c:main-Stanley-not-SP} implies that the Stanley
inequality~\eqref{eq:Sta} most likely cannot be proved by a direct injection.
This confirm an old speculation:


\begin{center}\begin{minipage}{12.cm}%
{{\em ``It appears unlikely that Stanley's Theorem for linear extensions quoted earlier
can be proved using the kind of injection presented here.''} \cite[$\S$4]{DDP}.
}
\end{minipage}\end{center}


\nin
Similarly, Corollary~\ref{c:main-Stanley-not-SP} suggests that the strategy
in \cite[$\S$9.12]{CPP-effective} is unlikely to succeed, at least for \ts
$k\ge 2$.\footnote{In~\cite[p.~129]{Gra}, Graham asked if Stanley's inequality
can be proved using the AD and FKG inequalities.  This seems unlikely,
even though we do not know how to formalize this question.}

To fully appreciate how delicate is Corollary~\ref{c:main-Stanley-not-SP},
compare it with a closely related problem.
It is known that for all \ts $k\ge 0$, the analogue of the Stanley
inequality~\eqref{eq:Sta} holds for the number \ts $\Om(P,t)$ \ts
of \defng{order preserving maps}
\. $X\to [t]$, for all \ts $t \in \nn$.  This was conjectured by Graham
in \cite[p.~129]{Gra82} (see also \cite[p.~233]{Gra}), motivated by the
proof of the XYZ inequality~\cite{She-XYZ} (cf.~$\S$\ref{ss:hist-LE}).
The result was proved in
\cite[Thm~4]{DDP} by a direct injection (see also \cite[$\S$4.2]{Day84}
for additional details of the proof).  In other words, in contrast with
$\ts \phi_k\ts$, the defect of the analogue of \eqref{eq:Sta} for order
preserving maps has a combinatorial interpretation.  Note that it is not
known whether the defect of the XYZ inequality is in \ts $\SP$,
see \cite[Conj.~6.4]{Pak-OPAC}.

\subsection{Stability proofs}\label{ss:finrem-mass-transport}
By analogy with the injective proofs, Corollary~\ref{c:main-AF-stab} suggests
that certain proofs of the Alexandrov--Fenchel inequality are likely not possible.
Here we are thinking of the mass transportation proof of characterization
of the isoperimetric sets given in \cite[App.]{FMP10}, following Gromov's
approach in~\cite{Gromov}.  It would be interesting to make this idea precise.

\subsection{Dichotomy of the equality cases}\label{ss:finrem-eq}
As we discuss in~$\S$\ref{ss:Sta1-critical}, it follows from the results
in~\cite{MS22}, that the equality verification of the Stanley inequality
\eqref{eq:Sta} can be decided in polynomial time for supercritical
posets.  In contrast, by Theorem~\ref{t:main-Sta}, the problem is not in
\ts $\PH$ \ts for critical posets.\footnote{We further clarify this in
our survey \cite[$\S$10]{CP23-survey}, written after this paper. }
We believe that this dichotomy also holds for the equality cases of the
Alexandrov--Fenchel inequality \eqref{eq:AF} for classes
of H-polytopes for which the scaled mixed volume is in~$\ts\SP$.

\smallskip

\subsection{The meaning of it all}\label{ss:finrem-meaning}
Finding the equality conditions of an inequality may seem like a
straightforward unambiguous problem, but the case of the Alexandrov--Fenchel
inequality shows that it is nothing of the kind.
Even the words ``equality conditions'' are much too vague for our
taste.  What the problem asks is a \defng{description of the equality cases}.
But since many geometric and combinatorial inequalities have large families
of equalities cases, the word ``description'' becomes open-ended
(cf.~$\S$\ref{ss:def-term}).  How do you know when you are done?
At what point are you satisfied with the solution and do not need
further details?

These are difficult questions which took many decades to settle, and the
answers depend heavily on the area.  In the context of geometric inequalities
discussed in~$\S$\ref{ss:hist-geom}, the meaning of ``description'' starts
out simple enough.  There is nothing ambiguous about discs as equality
cases of the isoperimetric inequality in the plane~\eqref{eq:Isop},
or pairs of homothetic convex bodies for the Brunn--Minkowski inequality~\eqref{eq:BM},
or circumscribed polygons with given normals for the discrete isoperimetric inequality
(see~$\S$\ref{ss:finrem-discrete-isop}).  Arguably, Bol's equality
cases of~\eqref{eq:Mink-mean} are also unambiguous --- in~$\rr^3$,
you literally know the cap bodies when you see them.

However, when it comes to Minkowski's quadratic inequality \eqref{eq:MQI},
the exact meaning of ``description'' is no longer obvious.  Shenfeld and
van~Handel write ``The main results of this paper will provide a complete
solution to this problem'' \cite{SvH-duke}.  Indeed, their description
of $3$-dimensional triples of convex bodies cannot be easily improved
upon, at least not in the case of convex polytopes (see~$\S$\ref{ss:hist-geom}).
Some questions may still linger, but they are on the structure of the equality
cases rather than on their recognition.\footnote{For example,
one can ask to characterize all possible triples of polytope graphs
that arise as equality cases.}

What Shenfeld and van~Handel did, is finished off the geometric approach
going back to Brunn, Minkowski, Favard, Fenchel, Alexandrov and others,
further formalized by Schneider.  ``Maybe a published conjecture will
stimulate further study of this question'', Schneider wrote in \cite{Schn85}.
This was prophetic, but that conjecture was not the whole story, as it turned out.

In \cite{SvH-acta}, the authors write again:  ``We completely settle
the extremals of the Alexandrov--Fenchel inequality for convex polytopes.''
Unfortunately, their description is extraordinary complicated in higher
dimensions, so the problem of \defng{recognizing} \ts the equality cases
is no longer easy (see~$\S$\ref{ss:hist-AF}).  And what good is a
description if it cannot be used to recognize the equality cases?

In combinatorics, the issue of ``description'' has also been a major problem
for decades, until it was fully resolved with the advent of computational complexity.
For example, consider the following misleadingly simple description: ``Let $G$ be a planar
cubic Hamiltonian graph.''  Is that good enough?  How can you tell if a given
graph~$G$ is as you describe?  We now know that the problem whether $G$ is planar,
cubic and Hamiltonian is \ts $\NP$-complete \cite{GJT}.  But if you only need
the ``planar'' condition, the problem is computationally easy, while the ``cubic'' condition
is trivial.  Consequently, ``planar cubic Hamiltonian'' should not be viewed as a
``good'' description, but if one must consider the whole class of such graphs,
this description is (most likely) the best one can do.

Going over equality cases for various inequalities on the numbers
of linear extensions, already gives an interesting picture.
For the Bj\"orner--Wachs inequality (see~$\S$\ref{ss:hist-LE}),
the recognition problem of forests is in~$\P$, of course.
On the other hand, as we explain in~$\S$\ref{ss:hist-LE}, for the Sidorenko
inequality~\eqref{eq:Sid1}, the recognition problem of series-parallel posets
is in~$\P$ for a more involved reason.  On the opposite end of the spectrum, for the
(rather artificial) inequality
\ts $(e(P)-e(Q))^2\ge 0$,  the equality verification is not
in \ts $\PH$, unless $\PH$ collapses, see~$\S$\ref{ss:hist-combin-int}
and \cite[Thm~1.4]{CP23}.

In this language, for the \ts $k=0$ \ts case of the Stanley inequality~\eqref{eq:Sta},
the description of equality cases given in \cite{SvH-acta} is trivially in~$\P$.
Similarly, for the \ts $k=1$ \ts case, the description of equality cases is
also in~$\P$  by Theorem~\ref{t:ESta-1}.  On the other hand, Theorem~\ref{t:main-Sta}
shows that for \ts $k\ge 2$,  the description in \cite{MS22} is (very likely)
not in~$\P$.  Under standard complexity assumptions, there is no description of the
equality cases in \ts $\P$ \ts at all, or even in \ts $\PH$ \ts for that matter.

Now, the problem of \emph{counting} \ts the equality cases brings a host of
new computational difficulties, making seemingly easy problems appear hard
when formalized, see~\cite{Pak-OPAC}.  Even for counting
non-isomorphic forest posets on $n$ elements, to show that this function
in $\SP$ one needs to define a \defng{canonical labeling} \ts to be able
to distinguish the forests, to make sure each is counted exactly once,
see e.g.~\cite{SW19}.

In this language, Corollary~\ref{c:main-Stanley-not-SP}
states that \emph{there are no combinatorial objects} \ts that can be counted
to give the number of non-equality cases of the Stanley inequality,
neither the non-equality cases themselves nor anything else.  The same applies
to the equality cases. Fundamentally, this
is because you should not be able to efficiently tell if the instances you are
observing are the ones you should be counting.

Back to the Alexandrov--Fenchel inequality~\eqref{eq:AF}, the description of
equality cases by Shenfeld and van Handel is a breakthrough in convex geometry,
and gives a complete solution for a large family of ($n$-tuples of)
convex polytopes (see~$\S$\ref{ss:finrem-eq}).
However, our Theorem~\ref{t:main-AF} says that from the computational point
of view, the equality cases are intractable in full generality.  Colloquially,
this says that \defna{there is no good description} \ts of the equality cases
of the Alexandrov--Fenchel inequality,
unless the world of computational complexity is not what we think it is.
As negative as this may seem, this is what we call a complete solution indeed.


\vskip.6cm

\subsection*{Acknowledgements}
We are grateful to Karim Adiprasito, Sasha Barvinok, K\'aroly B\"or\"oczky,
Christian Ikenmeyer, Jeff Kahn, Joe O'Rourke, Aldo Pratelli,
Matvey Soloviev and Richard Stanley for useful remarks on the subject.
Special thanks to Yair Shenfeld and Ramon van Handel for many
very helpful comments on the first draft of the paper, and to
Greta Panova for the numerous helpful discussions.

An extended abstract
of this paper is to appear in \emph{Proc.\ STOC} (2024); we thank
the Program Committee and the reviewers for helpful comments.
These results were obtained when both authors were visiting the American
Institute of Mathematics at their new location in Pasadena, CA.
We are grateful to AIM for the hospitality.  The first
author was partially supported by the Simons Foundation.
Both authors were partially supported by the~NSF.

}


\vskip1.1cm


%
%
%

\end{document}